\documentclass[a4paper,10pt]{article}

\usepackage{graphicx}
\usepackage{geometry}
\usepackage{amsthm}
\usepackage{amssymb}
\usepackage{amsmath}
\usepackage{hyperref}
\usepackage{xcolor}
\usepackage{enumerate}
\usepackage{listings}
\usepackage{complexity}
\usepackage{blkarray}
\usepackage{lmodern}
\usepackage[english]{babel}
\usepackage{accents}

\usepackage{tikz}
\usetikzlibrary{calc}
\usetikzlibrary{decorations.markings}
\usetikzlibrary{decorations.pathmorphing}
\usetikzlibrary{decorations.pathreplacing}

\usepackage{float}
\usepackage{caption}
\usepackage{subcaption}
\usepackage{graphicx}

\usepackage[font=small,labelfont=bf]{caption}
\usepackage[font=small,labelfont=normalfont,labelformat=simple]{subcaption}
\graphicspath{{figs/}}

\newtheorem{theorem}{Theorem}
\newtheorem{definition}[theorem]{Definition}
\newtheorem{lemma}{Lemma}[section]

\newtheorem{observation}[theorem]{Observation}
\newtheorem{proposition}[theorem]{Proposition}
\newtheorem{conjecture}[theorem]{Conjecture}
\newtheorem{problem}[theorem]{Problem}

\newcommand{\revvec}[1]{\accentset{\leftarrow}{#1}}

\newcommand{\bivectwo}[1]{\overset{\text{\tiny$\longleftrightarrow$}}{#1}}
\newcommand{\bivec}[1]{\accentset{\leftrightarrow}{#1}}

\newcommand{\dgirth}{\vec{g}}

\author
{
Lior Gishboliner\thanks{School of Mathematical Sciences, Sackler Faculty of Exact Sciences, Tel-Aviv University, Israel, email: \texttt{liorgis1@mail.tau.ac.il}. Research supported by ERC Starting Grant 633509.}
\and
Raphael Steiner\thanks{Institute of Mathematics, Technische Universit\"at Berlin, Germany, email: \texttt{steiner@math.tu-berlin.de}.
Funded by DFG-GRK 2434 Facets of Complexity.}
\and
Tibor Szab\'{o}\thanks{Institute of Mathematics, Freie Universit\"{a}t Berlin, Germany, email: \texttt{szabo@math.fu-berlin.de}. 
Research supported in part by GIF grant No. G-1347-304.6/2016 and by the Deutsche Forschungsgemeinschaft (DFG, German Research Foundation) under Germany's Excellence Strategy - The Berlin Mathematics Research Center MATH+ (EXC-2046/1, project ID: 390685689).}
}

\date{\today}

\title{Oriented cycles in digraphs of large outdegree}

\begin{document}
\maketitle

%\begin{abstract}
%In 1985, Mader conjectured the existence of a function $f:\mathbb{N} \rightarrow \mathbb{N}$ such that for every $k \in \mathbb{N}$, every digraph $D$ with minimum out-degree at least $f(k)$ contains a subdivision of $\vec{K}_k$, the transitive tournament of order $k$. This conjecture remains widely open, as even the existence of $f(5)$ is unknown. In view of the remarkable difficulty of Mader's question, Aboulker, Cohen, Havet, Lochet, Moura and Thomass\'{e} recently proposed the following two weaker conjectures:
%\begin{enumerate}
%\item For every $\ell \geq 2$ there is $K = K(\ell)$ such that every digraph $D$ with $\delta^+(D) \geq K$ contains a subdivision of every orientation of a cycle of length $\ell$.  
%\item For every oriented forest $F$ there exists a constant $K=K(F)$ such that every digraph $D$ with $\delta^+(D) \ge K$ contains a subdivision of $F$.
%\end{enumerate}
%In this paper we prove Conjecture 1 and give answers to further open problems raised by Aboulker et al. 
%\end{abstract}

\begin{abstract}
	In 1985, Mader conjectured that for every acyclic digraph $F$ there exists $K=K(F)$ such that every digraph $D$ with minimum out-degree at least $K$ contains a subdivision of $F$. This conjecture remains widely open, even for digraphs $F$ on five vertices. Recently, Aboulker, Cohen, Havet, Lochet, Moura and Thomass\'{e} studied special cases of Mader's problem and made the following conjecture:
	for every $\ell \geq 2$ there exists $K = K(\ell)$ such that every digraph $D$ with minimum out-degree at least $K$ contains a subdivision of every orientation of a cycle of length $\ell$. 

	We prove this conjecture and answer further open questions raised by Aboulker et al.
\end{abstract}

\section{Introduction}
A {\em subdivision} of a graph $F$ is a graph obtained from $F$ by replacing its edges with internally vertex-disjoint paths. 
This notion appears in some of the most fundamental results of graph theory, such as Kuratowski's characterization of planar graphs, as well as many classical results in the structure theory of sparse graphs. Because of these applications, it is desirable to understand by which means a given graph $G$ can be forced to contain a subdivision of a given graph $F$.
One such direction of study that has received a great amount of attention in the literature is the question of how ``dense'' $G$ should be to guarantee a subdivided $F$.
For undirected graphs, this problem has been solved with great precision. Mader~\cite{maderundirected} was the first to prove that for every fixed $k \in \mathbb{N}$, every graph of sufficiently large average degree contains a subdivision of $K_k$, and hence also of any other graph on at most $k$ vertices. The precise asymptotic dependence of the average degree on $k$, that is required to force $K_k$ as a subdivision, was independently determined by Bollob\'{a}s and Thomason~\cite{bollobas} and by Koml\'{o}s and Szemer\'{e}di~\cite{Komlos}.
\begin{theorem}[\cite{bollobas,Komlos}]\label{thm:undirected_clique}
	There is an absolute constant $C>0$ such that every graph with average degree at least $Ck^2$ contains a subdivision of $K_k$. This bound is best-possible up to the value of \nolinebreak $C$.
\end{theorem}
There is a natural analogue of subdivisions in directed graphs. Given a digraph $F$, a \emph{subdivision} of $F$ is a digraph obtained by replacing every arc $(x,y)$ in $F$ by a directed path from $x$ to $y$, such that subdivision-paths corresponding to different arcs are internally vertex-disjoint. It is natural to ask to what extent the above phenomenon, that every ''sufficiently dense`` graph contains a subdivision of a fixed graph $F$, extends to digraphs. 

Aboulker et al.~\cite{aboulker} introduced the following handy terminology for the study of forcing subdivisions of digraphs through various digraph parameters. Given a digraph parameter $\gamma$ ranging in $\mathbb{N}$, a digraph $F$ is called \emph{$\gamma$-maderian} if there exists a (smallest) number $\text{mader}_\gamma(F) \in \mathbb{N}$ such that every digraph $D$ with $\gamma(D) \ge \text{mader}_\gamma(F)$ contains a subdivision of $F$ as a subdigraph. We call $\text{mader}_\gamma(F)$ the {\em Mader number} of $F$ (with respect \nolinebreak to \nolinebreak $\gamma$). 

For example, using the natural analogue of these notions for undirected graphs, Theorem~\ref{thm:undirected_clique} states that the Mader number of $K_k$ with respect to the graph parameter $\bar{d}$, namely the average degree, is quadratic in $k$,
and in particular every graph $F$ is $\bar{d}$-maderian.

The \emph{average out-degree} (or, equivalently, {\em average in-degree}) of a digraph $D$ is $\overline{d}(D):=\frac{|A(D)|}{|V(D)|}$.
As the transitive tournament is a digraph of very high average out-degree which does not even contain a subdivision of a directed cycle, it should be clear that an analogue of 
Theorem~\ref{thm:undirected_clique} for digraphs cannot hold in its full generality. It turns out that the family of $\overline{d}$-maderian digraphs is limited to the so-called {\em anti-directed forest}s: forests in which every vertex is a sink or a source.
The positive direction of this result is the consequence of a theorem of Burr~\cite{burr2}, who proved that every digraph of sufficiently large average degree contains every anti-directed forest as a subgraph (and hence also as a subdivision). 
The negative direction, as pointed out by Aboulker et al.~\cite{aboulker}, follows by considering dense bipartite graphs of large girth and orienting all their edges from one side of the bipartition to the other.

The above constructions of dense digraphs without certain subdivisions all contain sinks (i.e. vertices of out-degree zero); this motivates the study of subdivisions in digraphs with large minimum out-degree. 
The \emph{minimum out-degree} ({\em minimum in-degree}) of a digraph $D$ is denoted by $\delta^+(D)$ (respectively, $\delta^-(D)$).

Since $\delta^+\leq \overline{d}$, every $\overline{d}$-maderian digraph is obviously also $\delta^+$-maderian. However, a characterization of $\delta^+$-maderian digraphs is still widely unknown.
Thomassen~\cite{thom85}, answering a question of Seymour in the negative, constructed digraphs of arbitrarily large minimum out-degree not containing directed cycles of even length. As a consequence, if a digraph $F$ has the property that each of its subdivisions contains a directed cycle of even length, then $F$ is not $\delta^+$-maderian. Digraphs with this property are known in the literature as \emph{even digraphs}, and have been extensively studied due to their relation to the so-called \emph{even cycle problem}. We refer the reader to 
%\cite{pfaffian},~\cite{seymthom},~\cite{thom85},~\cite{thom86},~\cite{thom92}
\cite{pfaffian,seymthom,thom85,thom86,thom92} 
for a selection of relevant literature. As can easily be verified by hand, the smallest even digraph is the bioriented clique $\bivec{K}_3$ of order $3$. This is also the smallest non-$\delta^+$-maderian digraph; indeed, 
%in Section~\ref{sec:k3-e} 
the following theorem states that $\bivec{K}_3-e$, the digraph obtained from $\bivec{K}_3$ by removing a single arc, is $\delta^+$-maderian. The proof of this theorem appears in Section~\ref{sec:k3-e}. 
\begin{theorem}\label{thm:K3-e}
	Every digraph $D$ with $\delta^+(D) \ge 2$ contains a subdivision of $\bivec{K}_3-e$.
\end{theorem}
Observe that for every digraph $F$ it holds that $\text{mader}_{\delta^+}(F) \geq |V(F)| - 1$, since the bioriented clique on $|V(F)| - 1$ vertices has minimum out-degree $|V(F)|-2$ but no subdivision of $F$. Hence, the bound in Theorem \ref{thm:K3-e} is optimal. 

Theorem \ref{thm:K3-e} strengthens an earlier result by Thomassen (cf.~\cite{thom2cycles}, Theorem 6.2), who proved that every digraph of minimum out-degree $2$ contains two directed cycles sharing precisely one vertex (this configuration is present in every subdivision of $\bivec{K}_3-e$). On the negative side, another construction by Thomassen~\cite{thom85} shows that there are digraphs of arbitrarily high minimum out-degree having no three directed cycles which share exactly one common vertex (and are otherwise disjoint). In other words, the bioriented $3$-star $\bivec{S}_3$ is not $\delta^+$-maderian. This result is somewhat surprising when compared to another positive result of Thomassen~\cite{thom83}, which shows that for every $k \in \mathbb{N}$ the digraph $k\bivec{K}_2$ (i.e., the disjoint union of $k$ digons) is $\delta^+$-maderian. More concretely, Thomassen proved that for every $k \in \mathbb{N}$ we have $\text{mader}_{\delta^+}(k\bivec{K}_2) \le (k+1)!$. The first linear bound on $\text{mader}_{\delta^+}(k\bivec{K}_2)$ was proven by Alon~\cite{alon}, and then further improved by Buci\'{c}~\cite{bucic}. The famous Bermond-Thomassen conjecture states that in fact $\text{mader}_{\delta^+}(k\bivec{K}_2)=2k-1$, but this remains widely open. 

A further negative result was established by DeVos et al.~\cite{DMMS}. Building on previous work of Mader~\cite{MaderDegLocalConn}, they constructed digraphs of arbitrarily high minimum out-degree having no pair of vertices $x,y$ with two arc-disjoint dipaths from $x$ to $y$ as well as two from $y$ to $x$ (see~\cite[Observation 8]{DMMS}). This result shows that every $\delta^+$-maderian digraph $F$ has arc-connectivity $\kappa'(F) \le 1$. On the positive side, Aboulker et al.~\cite{aboulker} proved that if $F$ is a digraph consisting of two vertices $x$ and $y$ and three internally vertex-disjoint dipaths between $x$ and $y$ -- two from $x$ to $y$ and one from $y$ to $x$ -- then $F$ is $\delta^+$-maderian.

The negative results discussed so far show that digraphs $F$ with a sufficiently rich directed cycle structure are not $\delta^+$-maderian. However, to this date, no  \emph{acyclic} digraph is known that is not $\delta^+$-maderian. This lead Mader~\cite{MaderDegLocalConn} to the following conjecture.
\begin{conjecture}[Mader, 1985]\label{conj:Mader}
	Every acyclic digraph is $\delta^+$-maderian.
\end{conjecture}
Clearly, it would suffice to prove Mader's conjecture for the transitive tournaments $\vec{K}_k$. Mader~\cite{Mader_trans_4} proved that $\text{mader}_{\delta^+}(\vec{K}_4)=3$, but the existence of $\text{mader}_{\delta^+}(\vec{K}_k)$ remains unknown for any $k \ge 5$. In view of the apparent difficulty of Mader's question, it is natural to try and verify Mader's conjecture for subclasses of acyclic digraphs. Mader himself~\cite{MaderLocalConnectivity} considered the digraph consisting of two vertices $x$ and $y$ and $k$ dipaths of length two from $x$ to $y$, and showed that it is $\delta^+$-maderian for all $k \in \mathbb{N}$. Aboulker et al.~\cite{aboulker} proposed to study the following two special cases of Mader's conjecture:
\begin{conjecture}[\cite{aboulker}]\label{con:forests}
	Every orientation of a forest is $\delta^+$-maderian.
\end{conjecture}
\begin{conjecture}[\cite{aboulker}]\label{con:cycles}
	Every orientation of a cycle is $\delta^+$-maderian.
\end{conjecture}
Aboulker et al. \cite{aboulker} proved two special cases of Conjecture~\ref{con:forests}, showing that every orientation of a path and every in-arborescence (an oriented tree with all edges directed towards a specified root) is $\delta^+$-maderian. They also proved Conjecture~\ref{con:cycles} for oriented cycles consisting of two {\em blocks}\footnote{By a {\em block} in an oriented cycle we mean a maximal directed subpath.}, i.e., oriented cycles having exactly one source and one sink. 

Our main contribution in this paper is to verify Conjecture~\ref{con:cycles} in its full generality. Moreover, we show that the Mader number $\text{mader}_{\delta^+}$ of an oriented cycle grows (only) polynomially with the cycle length. Let $C_{\ell}$ denote the undirected cycle of length $\ell$.  
\begin{theorem}\label{thm:cycles}
	There exists a polynomial function $K:\mathbb{N} \rightarrow \mathbb{N}$ such that for every $\ell \ge 2$, every digraph $D$ with $\delta^+(D) \ge K(\ell)$ contains a subdivision of every orientation of $C_\ell$.
\end{theorem}
\noindent
The proof of Theorem~\ref{thm:cycles} is presented in Section~\ref{sec:mindegreeorientedcycles}. 

Let $k_1, k_2 \in \mathbb{N}$. Following the notation in~\cite{aboulker}, we denote by $C(k_1,k_2)$ the two-block cycle consisting of two vertices $x,y$ and two internally vertex-disjoint dipaths from $x$ to $y$ of length $k_1$ and $k_2$, respectively. 
%Considering the bidirected clique of order $k_1+k_2-1$, which contains no subdivision of $C(k_1,k_2)$ but has minimum out-degree $k_1+k_2-2$, we obtain the trivial lower bound of $\text{mader}_{\delta^+}(C(k_1,k_2)) \ge k_1+k_2-1$. 
As mentioned above, we have the trivial lower bound $\text{mader}_{\delta^+}(C(k_1,k_2)) \ge k_1+k_2-1$. 
Aboulker et al. (see \cite[Theorem~24]{aboulker}) proved the upper bound $\text{mader}_{\delta^+}(C(k_1,k_2)) \le 2(k_1+k_2)-1$. They also observed that the trivial lower bound gives the truth if $k_2 = 1$, showing that $\text{mader}_{\delta^+}(C(k,1)) = k$ for every $k \geq 1$. They then asked whether or not their aforementioned bound on $\text{mader}_{\delta^+}(C(k_1,k_2))$ is tight.
\begin{problem}[\cite{aboulker}, Problem 25]
	For $k_1, k_2 \ge 1$, what is the value of $\text{mader}_{\delta^+}(C(k_1,k_2))$?
\end{problem}
%In Section~\ref{sec:twoblocks} we show that the bound by Aboulker et al.~\cite{aboulker} is not tight, by proving an upper bound on $\text{mader}_{\delta^+}(C(k_1,k_2))$ which is strictly smaller than the bound given in \cite{aboulker} for all values of $k_1, k_2 \ge 2$, and is asymptotically better if $k_1\gg k_2$. Our bound is tight for $k_2=2$, showing that $\text{mader}_{\delta^+}(C(k,2)) = k+1$ for every $k \ge 1$.
\noindent
Our next result improves upon the bound given by Aboulker et al.~\cite{aboulker}.
\begin{theorem}\label{thm:betterbound}
	Let $k_1 \ge k_2 \ge 2$ be integers. Then $\text{mader}_{\delta^+}(C(k_1,k_2)) \le k_1+3k_2-5$.
\end{theorem}
Theorem \ref{thm:betterbound} improves upon the result of \cite{aboulker} for all values of $k_1,k_2 \geq 2$, and is asymptotically better if $k_1\gg k_2$. Furthermore, if $k_2 = 2$ then the bound in Theorem \ref{thm:betterbound} is optimal, as it matches the aforementioned trivial lower bound, thus showing that $\text{mader}_{\delta^+}(C(k,2)) = k+1$ for every $k \ge 1$. 
%This shows that the trivial lower bound is tight in the case $k_2 = 2$ as well. 
The proof of Theorem \ref{thm:betterbound} appears in Section~\ref{sec:twoblocks}.

%As a second result, in Section~\ref{sec:forests} we prove Conjecture~\ref{con:forests} for a large family of oriented forests, namely for all oriented {\em subcubic} forests without sinks and sources of degree $3$. Here, we say that a digraph is subcubic if every vertex $v$ has undirected degree at most $3$, i.e. $d^+(v) + d^-(v) \leq 3$. 
%\begin{theorem}\label{thm:forests}
%Let $F$ be an oriented subcubic forest such that $\Delta^+(F), \Delta^-(F) \le 2$. Then $F$ is $\delta^+$-maderian. 
%\end{theorem}
To conclude, let us mention that in contrast to the aforementioned negative results for general directed graphs, if we restrict our attention to the class of tournaments, which have an inherent density property, then it can be proved that every digraph is forcible as a subdivision by means of large minimum out-degree. This is a recent result by Gir\~{a}o, Popielarz and Snyder~\cite{GPS}, which in addition gives a best-possible asymptotic bound of $Ck^2$ on the minimum out-degree of a tournament required to guarantee the existence of a subdivision of the bioriented $k$-clique.

As the family of $\delta^+$-maderian digraphs is still somewhat limited, Aboulker et al.~\cite{aboulker} initiated the study of the effect of even stronger density conditions, involing the
\emph{strong vertex-connectivity} $\kappa$, and the \emph{strong arc-connectivity} $\kappa'$
of digraphs. Since $\kappa \leq \kappa'\leq \delta^+$, every $\delta^+$-maderian digraph is obviously $\kappa'$- and $\kappa$-maderian. 
Not much is known however concerning how much richer the families of $\kappa$- and $\kappa'$-maderian digraphs are. The following interesting questions were posed in~\cite{aboulker}:
\begin{problem}[\cite{aboulker}, Problem 16]\label{prob:connectivity}
	Is every digraph $\kappa$-maderian? Is every digraph $\kappa'$-maderian?
\end{problem}
While the first question remains open, we can resolve the second question in the negative by proving that neither the bioriented $4$-clique $\bivec{K}_4$ nor the bioriented $4$-star $\bivec{S}_4$ is $\kappa'$-maderian: 
\begin{proposition}\label{prop:arc_connectivity}
	For every $k \in \mathbb{N}$, there exists a digraph $G_k$ with $\kappa'(G_k) \ge k$ such that $G_k$ contains no subdivision of $\bivec{K}_4$.
\end{proposition}
\begin{proposition}\label{prop:arc_connectivity2}
	For every $k \in \mathbb{N}$, there exists a digraph $H_k$ with $\kappa'(H_k) \ge k$ such that $H_k$ contains no subdivision of $\bivec{S}_4$.
\end{proposition}
\noindent
The proofs of Propositions~\ref{prop:arc_connectivity} and~\ref{prop:arc_connectivity2} are presented in Section~\ref{sec:arcconn}. 

We note that a main difficulty arising when studying subdivisions in digraphs (as opposed to undirected graphs) is that  digraphs of large (strong) vertex-connectivity may not be linked. Recall that a digraph is called {\em $k$-linked} if for every $2k$-tuple of distinct vertices $x_1,\dots,x_k,y_1,\dots,y_k$, there are vertex-disjoint dipaths $P_1,\dots,P_k$ such that $P_i$ goes from $x_i$ to $y_i$. In undirected graphs, it is known that a graph with sufficiently large vertex-connectivity is $k$-linked (see \cite{bollobas_linked}), and linkedness has proven very useful for embedding subdivisions. In stark contrast, a construction of Thomassen~\cite{thom91} shows that for every $k \in \mathbb{N}$ there is a strongly $k$-vertex-connected digraph which is not $2$-linked. This makes subdivision questions for digraphs significantly more challenging.

\paragraph*{Notation.}
Digraphs in this paper are considered loopless, have no parallel edges, but are allowed to have anti-parallel pairs of edges (\emph{digons}). A directed edge (also called arc) with tail $u$ and head $v$ is denoted by $(u,v)$. For a graph $G$, we denote by $V(G)$, $E(G)$ its vertex- and edge-set, respectively. Similarly, the vertex-set (resp. arc-set) of a digraph $D$ is denoted by $V(D)$ (resp. $A(D)$). For $X \subseteq V(D)$, we denote by $D[X]$ the  subdigraph of $D$ induced by $X$. For a set $X$ of vertices or arcs in $D$, we denote by $D-X$ the subdigraph obtained by deleting the objects in $X$ from $D$.
%Given an undirected simple graph $G$, an \emph{orientation} of $G$ is any digon-free digraph with vertex set $V(G)$ in which two vertices are adjacent iff they are adjacent in $G$. 
Given an undirected simple graph $G$, an \emph{orientation} of $G$ is any digraph obtained by replacing each edge $\{u,v\}$ of $G$ with (exactly) one of the arcs $(u,v)$ or $(v,u)$. Evidently, any orientation is digon-free. 
%The \emph{biorientation} $\bivec{G}$ is defined as the digraph with vertex-set $V(G)$ which contains all symmetric arcs, i.e., $A(\bivec{G}):=\{(u,v),(v,u)|uv \in E(G)\}$.
The \emph{biorientation} $\bivec{G}$ is defined as the digraph obtained from $G$ by replacing every edge with a digon, i.e. $A(\bivec{G}):=\{(u,v),(v,u) \; | \; \{u,v\} \in E(G)\}$. 
We use $K_k$ to denote the $k$-clique, $C_k$ to denote the cycle of length $k$, and $S_k$ to denote the star with $k$ edges (all of these notations are for undirected graphs). 
%For $k \in \mathbb{N}$, the \emph{complete digraph of order $k$} or \emph{bidirected $k$-clique} $\bivec{K}_k$ is the biorientation of $K_k$. The \emph{bidirected $k$-star} $\bivec{S}_k$ is the biorientation of the $k$-star, i.e., the graph consisting of $k$ leaves attached to a central vertex. 
For a digraph $D$ and a vertex $v \in V(D)$, we let $N^+(v),$ $N^-(v)$ denote the out- and in-neighborhood of $v$ in $D$ and $d^+(v)$, $d^-(v)$ their respective sizes. We denote by $\delta^+(D)$, $\delta^-(D)$, $\Delta^+(D)$, $\Delta^-(D)$ the minimum or maximum  out- or in-degree of $D$, respectively. 
A \emph{directed walk} in a digraph is an alternating sequence $v_1,e_1,v_2,\ldots,v_{k-1},e_{k-1},v_k$ of vertices and arcs such that $e_i=(v_i,v_{i+1})$ for all $1 \leq i \leq k-1$. The walk is called {\em closed} if $v_k = v_1$. 
We use the words ``path" and ``cycle" to mean an orientation of a path or a cycle (respectively).
For example, a \emph{path} $P$ in a digraph $D$ is an alternating sequence $v_1,e_1,v_2,\ldots,v_{k-1},e_{k-1},v_k$ of pairwise distinct vertices $v_1,\ldots,v_k \in V(D)$ and arcs $e_1,\ldots,e_{k-1} \in A(D)$, such that $e_i$ connects $v_{i}$ and $v_{i+1}$ (i.e., either $e_i = (v_i,v_{i+1})$ or $e_i = (v_{i+1},v_i)$). 
%With a slight abuse of notation, in this paper we identify paths in digraphs with their corresponding subdigraphs of $D$. 
If $e_i=(v_{i},v_{i+1})$ for every $i=1,\ldots,k-1$, then we say that $P$ is a \emph{directed path} or \emph{dipath} from $v_1$ to $v_k$ ({\em $v_1$-$v_k$-dipath} for short). We will call $v_1$ the first vertex of $P$, $v_2$ the second vertex of $P$, $v_k$ the last vertex of $P$, etc. 
We denote by $|P|$ the length of $P$ (i.e. its number of arcs). 
Given two distinct vertices $x \neq y$ on a path $P$, we denote by $P[x,y]=P[y,x]$ the subpath of $P$ with endpoints $x$ and $y$. A vertex $v$ in a digraph $D$ is said to be \emph{reachable} from a vertex $u$ if there exists a $u$-$v$-dipath. In this case, the \emph{distance} from $u$ to $v$ in $D$ is defined as the length of a shortest $u$-$v$-dipath.  For a pair of dipaths $P,Q$ such that the first vertex of $Q$ is the last vertex of $P$, we denote by $P \circ Q$ the concatenation of $P$ and $Q$, i.e. the directed walk obtained by first traversing $P$ and then traversing $Q$. When $P$ (resp. $Q$) consists of a single arc $(x,y)$, we will sometimes write $(x,y) \circ Q$ (resp. $P \circ (x,y)$) instead of $P \circ Q$. 
%A \emph{cycle} in $D$ is a subdigraph of $D$ isomorphic to an orientation of a cycle.
%By the word {\em cycle} we mean an orientation of a cycle.  
A \emph{directed cycle} is a cycle with all arcs oriented consistently in one direction. The {\em directed girth} of $D$, i.e. the minimum length of a directed cycle in $D$, is denoted by $\dgirth(D)$. For a {\em directed} cycle $C$ and two distinct vertices $x,y \in V(C)$, we denote by $C[x,y]$ the segment of $C$ which forms a dipath from $x$ to $y$. 
%A \emph{closed directed walk} in a digraph is an alternating sequence $v_0,e_0,v_1,\ldots,v_{k-1},e_{k-1},v_k=v_0$ of vertices and arcs such that $e_i=(v_i,v_{i+1})$ for all $i$ (addition modulo $k$).
A digraph $D$ is called \emph{weakly connected} (or just {\em connected}) if every two vertices can be connected by a path (i.e., if the underlying undirected graph is connected), and is called \emph{strongly connected} if for every ordered pair of vertices $(x,y) \in V(D) \times V(D)$, $x$ can reach $y$ in $D$. The maximal strongly connected subgraphs of a digraph $D$ are called \emph{(strong) components} and induce a partition of $V(D)$. For a natural number $k \in \mathbb{N}$, a digraph $D$ is called \emph{strongly $k$-vertex (arc)-connected} if for every set $K$ of at most $k-1$ vertices (arcs) of $D$, the digraph $D-K$ is strongly connected. An in- (resp., out-) arborescence is a directed rooted tree in which all arcs are directed towards (resp., away from) the root. 

\paragraph{Preliminaries.}
We now quickly recall Menger's Theorem, which will use in the course of the article.
The following is a well-known variant of Menger's Theorem for \nolinebreak directed \nolinebreak graphs.
\begin{theorem}[see \cite{mengeroriginal}]\label{vertexmenger}
Let $D$ be a digraph and $u, v \in V(D)$ be distinct vertices such that $(u,v) \notin A(D)$. Then for every $k \in \mathbb{N}$, either there are $k$ internally vertex-disjoint $u$-$v$-dipaths in $D$, or there is a set $K \subseteq V(D) \setminus \{u,v\}$ such that $|K| < k$ and $D-K$ contains no $u$-$v$-dipath.
\end{theorem}
Given a digraph $D$ and two (not necessarily disjoint) subsets $A, B \subseteq V(D)$, an \emph{$A$-$B$-dipath} is a directed path in $D$ which starts at a vertex of $A$, ends at a vertex of $B$, and is internally vertex-disjoint from $A \cup B$ 
%(here we allow paths consisting of a single vertex in $A \cap B$).
(an $A$-$B$-dipath is allowed to consist of a single vertex in $A \cap B$). 
%Similarly, for a vertex $u \in V(D)$, by a $u$-$A$-dipath or an $A$-$u$-dipath, respectively, we mean a $\{u\}$-$A$ or an $A$-$\{u\}$-dipath according to the above definition.
If $A$ or $B$ are of size one, say $A = \{u\}$ or $B = \{u\}$, then we will simply write ``$u$-$B$-dipath" or ``$A$-$u$-dipath", respectively.  
%Similarly, for a vertex $u \in V(D)$, by a $u$-$A$-dipath or an $A$-$u$-dipath, respectively, we mean a $\{u\}$-$A$ or an $A$-$\{u\}$-dipath according to the above definition. 
The following is a well-known consequence of Theorem~\ref{vertexmenger}.
\begin{theorem}\label{setmenger}
Let $D$ be a digraph, let $v \in V(D)$ and let $A \subseteq V(D) \setminus \{v\}$. Then either there are $k$ different $v$-$A$-dipaths which pairwise only intersect at $v$, or there is a subset $K \subseteq V(D) \setminus \{v\}$ such that $|K|<k$ and such that there is no dipath in $D-K$ starting in $v$ and ending in $A$. 
\end{theorem}
\begin{proof}
Consider the digraph $H$ obtained from $D$ by adding an artificial vertex $v_A \notin V(D)$ and adding the arc $(y,v_A)$ for every $y \in A$. The claim now follows by applying Theorem~\ref{vertexmenger} to the vertices $v$ and $v_A$ in $H$. Indeed, if there are $k$ internally vertex-disjoint $v$-$v_A$-dipaths in $H$, then by deleting all successors of the first vertex in $A$ from each of these dipaths (i.e. by cutting each of the dipaths as soon as it reaches $A$), we obtain $k$ distinct $v$-$A$-dipaths in $D$ which pairwise only share the vertex $v$. And if we can hit all $v$-$v_A$-dipaths in $H$ with a subset $K \subseteq V(H) \setminus \{v,v_A\}=V(D)\setminus \{v\}$ such that $|K|<k$, then there are no dipaths in $D-K$ starting in $v$ and ending in $A$. This proves the claim.
\end{proof}
%We will further need the following two deep results by Mader on so-called non-critical vertices and on subdivisions in digraphs of sufficiently large out-degree.
%\begin{theorem}[\cite{Mader_critical_connectivity}, see also Section 7.11 in~\cite{BJ-G}]\label{critical strongly connected}
%Let $k \in \mathbb{N}$, and let $D$ be a strongly $k$-vertex-connected digraph with $\delta^+(D),\delta^-(D) \geq 2k$. Then there is $v \in V(D)$ such that $D - v$ is still strongly $k$-vertex-connected. 
%\end{theorem}
%\begin{theorem}[\cite{Mader_trans_4}]\label{subdivtrans4}
%Let $D$ be a digraph such that $\delta^+(D) \ge 3$. Then $D$ contains a subgraph isomorphic to a subdivision of $\vec{K}_4$, the transitive tournament of order $4$.
%\end{theorem}
\section{Subdivisions of oriented cycles}\label{sec:mindegreeorientedcycles}
In this section, we prove Theorem~\ref{thm:cycles}, which we restate here for convenience.
% Our main result is as follows. 
\begin{theorem}\label{thm:cycle_orientation}
	For every $\ell \geq 2$ there is a polynomially bounded $K = K(\ell)$ such that every digraph $D$ with $\delta^+(D) \geq K$ contains a subdivision of every oriented cycle of length $\ell$.  
\end{theorem}

%To prove Theorem \ref{thm:cycle_orientation}, we first use the following result of Dellamonica, Koubek, Martin and R\"{o}dl \cite{DKMR} to reduce to the case that the host digraph $D$ has large (directed) girth. 
%\begin{theorem}[\cite{DKMR}]
%	For every $k \geq 1$ and $g \geq 3$ there exists $K = K(k,g)$ such that every digraph $D$ with $\delta^+(D) \geq K$ contains a subdigraph $D'$ with $\delta^+(D') \geq k$ and with directed girth at least $g$. 
%\end{theorem}

It is well-known and easy to show that every digraph with minimum out-degree $k$ contains a directed cycle of length at least $k+1$. Thus, in what follows we restrict our attention to acyclic oriented cycles. For integers $a,b \geq 1$, let $C_{a,b}$ be the oriented cycle consisting of $2a$ vertices $s_1,\dots,s_a,t_1,\dots,t_a$ and $2a$ internally-disjoint length-$b$ dipaths: one from $s_i$ to $t_i$ and one from $s_i$ to $t_{i+1}$ for each $1 \leq i \leq a$ (with indices taken modulo $a$). See Figure 1 for an illustration of $C_{2,3}$. It is easy to see that for every acyclic oriented cycle $C$, there are $a,b \geq 1$ such that every subdivision of $C_{a,b}$ is also a subdivision of $C$ (specifically, $a$ is the number of sources (or, equivalently, sinks) in $C$, and $b$ is the largest length of a dipath contained in $C$). Therefore, it is sufficient to show that digraphs with minimum out-degree at least $k(a,b)$ contain a subdivision of $C_{a,b}$ (for some suitable choice of $k(a,b) = \poly(a,b)$). For $a = 1$, this statement was proven in \cite{aboulker}, and we also give a new proof in Section~\ref{sec:twoblocks}. Consequently, it is sufficient to consider the case $a \geq 2$ (and, in fact, the assumption $a \geq 2$ is required by our method). 

\begin{figure}[h]\label{fig:b-block_cycle}
	\centering
	\begin{tikzpicture}[scale = 2]
	\foreach \i in {0,1,2,3,4,5,6,7,8,9,10,11,12}
	{
		\coordinate (x\i) at ({cos(\i*360/12)},{sin(\i*360/12)});
		\draw (x\i) node[fill=black,circle,minimum size=2pt,inner sep=0pt] {};
		\coordinate (y\i) at ({1.15*cos(\i*360/12)},{1.15*sin(\i*360/12)});
	} 
	
	%		\foreach \i in {0,3,6,9}
	%		{
	%			\draw (x\i) node[fill=black,circle,minimum size=3pt,inner sep=0pt] {};
	%		} 
	
	%		\foreach \i in {1,2,4,5,7,8,10,11}
	%		{
	%			\coordinate (x\i) at ({cos(\i*360/12)},{sin(\i*360/12)});
	%		} 
	%	
	%		\foreach \i in {0,3,6,9}
	%		{
	%			\coordinate (x\i) at ({1.25*cos(\i*360/12)},{1.25*sin(\i*360/12)});
	%			\coordinate (y\i) at ({1.5*cos(\i*360/12)},{1.5*sin(\i*360/12)});
	%		}
	%	
	%		\foreach \i in {0,1,2,3,4,5,6,7,8,9,10,11}
	%		{
	%			\draw (x\i) node[fill=black,circle,minimum size=2pt,inner sep=0pt] {};	
	%		} 
	
	%		\foreach \i in {0,3,6,9}
	%		{
	%			\draw (y\i) node {$s_{$\i/3$}$};
	%		} 
	\draw (y0) node {$s_{1}$};
	\draw (y3) node {$t_{1}$};
	\draw (y6) node {$s_{2}$};
	\draw (y9) node {$t_{2}$};
	
	\draw[decoration={markings,mark=at position 0.5 with {\arrow{>}}},postaction={decorate}] (x0) -- (x1);
	\draw[decoration={markings,mark=at position 0.5 with {\arrow{>}}},postaction={decorate}] (x1) -- (x2);
	\draw[decoration={markings,mark=at position 0.5 with {\arrow{>}}},postaction={decorate}] (x2) -- (x3);
	\draw[decoration={markings,mark=at position 0.5 with {\arrow{>}}},postaction={decorate}] (x6) -- (x5);
	\draw[decoration={markings,mark=at position 0.5 with {\arrow{>}}},postaction={decorate}] (x5) -- (x4);
	\draw[decoration={markings,mark=at position 0.5 with {\arrow{>}}},postaction={decorate}] (x4) -- (x3);
	\draw[decoration={markings,mark=at position 0.5 with {\arrow{>}}},postaction={decorate}] (x6) -- (x7);
	\draw[decoration={markings,mark=at position 0.5 with {\arrow{>}}},postaction={decorate}] (x7) -- (x8);
	\draw[decoration={markings,mark=at position 0.5 with {\arrow{>}}},postaction={decorate}] (x8) -- (x9);
	\draw[decoration={markings,mark=at position 0.5 with {\arrow{>}}},postaction={decorate}] (x0) -- (x11);
	\draw[decoration={markings,mark=at position 0.5 with {\arrow{>}}},postaction={decorate}] (x11) -- (x10);
	\draw[decoration={markings,mark=at position 0.5 with {\arrow{>}}},postaction={decorate}] (x10) -- (x9);
	\end{tikzpicture}
	\caption{The oriented cycle $C_{2,3}$}
\end{figure}
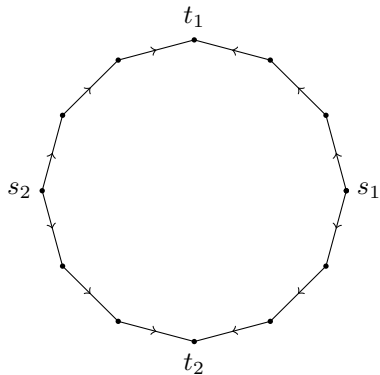

Dellamonica, Koubek, Martin and R\"{o}dl \cite{DKMR} proved that for every $k \geq 1$ and $g \geq 3$ there exists $K = K(k,g)$ such that every digraph $D$ with $\delta^+(D) \geq K$ contains a subdigraph $D'$ with $\delta^+(D') \geq k$ and with directed girth $\dgirth(D)$ at least $g$. Thus, in order to prove Theorem \ref{thm:cycle_orientation}, it suffices to establish the following: 

\begin{theorem}\label{thm:cycle_orientation_main}
There is an absolute constant $C$ such that for every pair of integers $a \geq 2, \; b \geq 1$, every digraph $D$ with $\delta^+(D) \geq Cab^7$ and $\dgirth(D) \geq 4b^2$ contains a subdivision of $C_{a,b}$.
\end{theorem}

A quantitative version\footnote{In fact, the bound on $K(k,g)$ appearing in \cite{DKMR} was slightly weaker --- in that the logarithmic factor depended on $k$ --- but it is easy to see that by using the argument of \cite{DKMR} and replacing a union bound used there with a tighter concentration inequality (say, Chernoff's bound), one obtains the stronger estimate stated here.} of the aforementioned result of \cite{DKMR} is that $K(k,g) \leq O(kg^2\log g)$. It follows that having minimum out-degree at least
$a \cdot \text{poly}(b)$ is enough to force a subdivision of $C_{a,b}$, and that the conclusion of Theorem \ref{thm:cycle_orientation} holds with $K(\ell) = \poly(\ell)$. 
%??? Express the bound in terms of $a$ and $b$. ???

For the rest of this section we set $g := 4b^2$. Our proof of Theorem \ref{thm:cycle_orientation_main} will use a certain structure we call a {\em chain}, that will consist of some carefully chosen {\em gadgets}. This structure will enable us to embed subdivisions of $C_{a,b}$ in the given digraph. We start by presenting these \nolinebreak key \nolinebreak definitions.

\subsection{The Gadgets}
We will use three types of gadgets.
Each of the gadgets will have a special pair of vertices $p,q$ with an arc from $p$ to $q$. The gadgets are defined as follows: 
\begin{enumerate}
%	\item[(I)] A directed cycle of length (exactly) $g$ through the arc $(p,q)$ is called a {\em gadget of type I}.
	\item[(I)] A gadget of type I is a directed cycle of length at least $g$ through the arc $(p,q)$. 
%	\item[(III)] A gadget of type III is a digraph consisting of vertices $u,v,w_1,\dots,w_{\ell}=z_1,z_2,\dots,z_{\ell}$ and arcs $(u,v)$; $(w_i,v)$ for every $1 \leq i \leq \ell$; $(w_i,w_{i-1})$ for every $2 \leq i \leq \ell$; $(w_1,u)$; $(z_i,w_{\ell-1})$ for every $1 \leq i \leq \ell$; $(z_i,z_{i-1})$ for every $2 \leq i \leq \ell$; and $(z_1,w_{\ell})$. 
	\item[(II)] 
	A {\em basic gadget of type II} is a digraph consisting of vertices $p,q,r$ and a dipath $P_1$ from $r$ to $p$, such that $P_1$ has length at least $2b^2 + b - 2$, $q \notin V(P)$, and every vertex of $P_1$ has an arc to $q$ (so in particular, $(p,q)$ is an arc). 
	An {\em extended gadget of type II} consists of a basic gadget of type II, comprised of vertices $p,q,r$ and a dipath $P_1$ as above, as well as an additional dipath $P_2$ of length at least $b$ having the following properties: 
	\begin{enumerate}
		\item The last vertex of $P_2$ is $r$, $V(P_1) \cap V(P_2) = \{r\}$, and $q \notin V(P_2)$. 
		\item Either there is an arc from the first vertex of $P_2$ to the second vertex of $P_1$, or there is an arc from some vertex in $V(P_1) \setminus \{r\}$ to the first vertex of $P_2$. 
	\end{enumerate}
%	For an extended type-II gadget $G$ consisting of dipaths $P_1,P_2$, the {\em basic part} of $G$ is the corresponding basic type-II gadget, namely the subgraph of $G$ induced by $V(P_1) \cup \{q\}$. 
		For an extended type-II gadget $G$, the {\em basic part} of $G$ is the corresponding basic type-II gadget, namely the subgraph of $G$ induced by $V(P_1) \cup \{q\}$. 
	\item[(III)] A gadget of type III is a digraph consisting of vertices $p,q,r$, the arc $(p,q)$, and two internally disjoint dipaths $P_1,P_2$ from $p$ and $q$, respectively, to $r$, such that $P_1$ and $P_2$ have length at least $2b-1$ each. 
\end{enumerate}  

\noindent
The various types of gadgets are depicted in Figures 2-3. For convenience, we also introduce the notion of a {\em trivial gadget}: a trivial gadget simply consists of vertices $p,q$ and the arc $(p,q)$ (and no other vertices). 

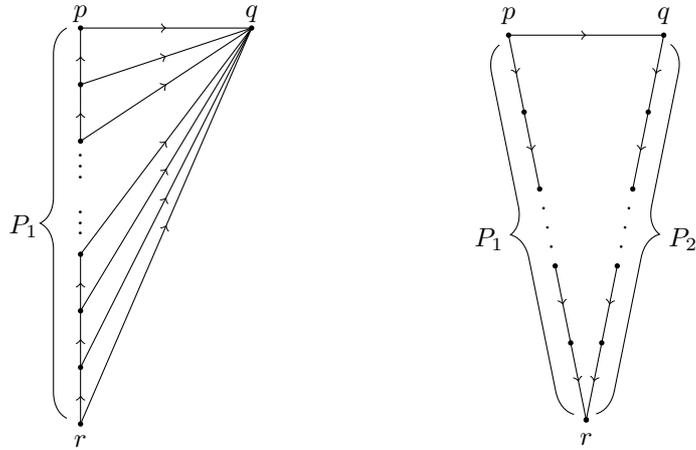
\begin{figure}\label{fig:type_II_gadget_basic}
	\begin{subfigure}{8.5cm}
	\centering
	\begin{tikzpicture}[scale = 0.75]
	\foreach \i in {0,1,2,3,4,5,6,7}
	{
		\coordinate (x\i) at (0,\i);
		% \draw (x\i) node[fill=black,circle,minimum size=2pt,inner sep=0pt] {};
	}
	
	\coordinate (q) at (3,7);
	\draw (q) node[fill=black,circle,minimum size=2pt,inner sep=0pt] {};
	
	\foreach \i in {0,1,2,3,5,6,7}
	{
		\draw (x\i) node[fill=black,circle,minimum size=2pt,inner sep=0pt] {};
		\draw[decoration={markings,mark=at position 0.5 with {\arrow{>}}},postaction={decorate}] (x\i) -- (q);
	}
	
	\foreach \i in {4,5}
	{
		\draw ($(x\i) + (0,-0.3)$) node[] {$\vdots$};
	}
	
	\draw[decoration={markings,mark=at position 0.5 with {\arrow{>}}},postaction={decorate}] (x0) -- (x1);
	\draw[decoration={markings,mark=at position 0.5 with {\arrow{>}}},postaction={decorate}] (x1) -- (x2);
	\draw[decoration={markings,mark=at position 0.5 with {\arrow{>}}},postaction={decorate}] (x2) -- (x3);
	%		\draw[decoration={markings,mark=at position 0.5 with {\arrow{>}}},postaction={decorate}] (x3) -- (x4);
	\draw[decoration={markings,mark=at position 0.5 with {\arrow{>}}},postaction={decorate}] (x5) -- (x6);
	\draw[decoration={markings,mark=at position 0.5 with {\arrow{>}}},postaction={decorate}] (x6) -- (x7);
	
	\draw ($(x7) + (0,0.25)$) node {$p$};
	\draw ($(q) + (0,0.25)$) node {$q$};
	\draw ($(x0) + (0,-0.3)$) node {$r$};
	
%	\draw [decorate,decoration={brace,amplitude=10pt},xshift=-4pt,yshift=0pt]
%	(-0.25,0) -- (-0.25,7) node {};
%	\draw (-1.25,3.5) node {$P_1$};
	
	\draw[decorate,decoration={brace,amplitude=10pt},xshift=-4pt,yshift=0pt]
	(-0.1,0.1) -- (-0.1,7) node {};
	\draw (-1,3.5) node {$P_1$};
	\end{tikzpicture}
%	\caption{A basic gadget of type II}
	\end{subfigure} 
	\begin{subfigure}{0cm}
		\centering
		\begin{tikzpicture}[scale = 1.02]
		\foreach \i in {0,1,2,3,4,5}
		{
			\coordinate (p\i) at ({-\i/5},\i); 
			\coordinate (q\i) at ({\i/5},\i);
			%			\draw (p\i) node[fill=black,circle,minimum size=2pt,inner sep=0pt] {};
			%			\draw (q\i) node[fill=black,circle,minimum size=2pt,inner sep=0pt] {};
		}
		
		\foreach \i in {0,1,2,3,4,5}
		{
			\draw (p\i) node[fill=black,circle,minimum size=2pt,inner sep=0pt] {};
			\draw (q\i) node[fill=black,circle,minimum size=2pt,inner sep=0pt] {};
		}
		
		\foreach \j in {1,2,3}
		{
			\coordinate (a) at ($0.25*\j*(p3)$);
			\coordinate (b) at ($0.25*4*(p2) - 0.25*\j*(p2)$);
			\coordinate (c) at ($(a) + (b)$); 
			\draw (c) node[fill=black,circle,minimum size=1pt,inner sep=0pt] {}; 
		}
		
		\foreach \j in {1,2,3}
		{
			\coordinate (a) at ($0.25*\j*(q3)$);
			\coordinate (b) at ($0.25*4*(q2) - 0.25*\j*(q2)$);
			\coordinate (c) at ($(a) + (b)$); 
			\draw (c) node[fill=black,circle,minimum size=1pt,inner sep=0pt] {}; 
		}
		
		\draw[decoration={markings,mark=at position 0.5 with {\arrow{>}}},postaction={decorate}] (p5) -- (q5);
		
		\draw[decoration={markings,mark=at position 0.5 with {\arrow{>}}},postaction={decorate}] (p5) -- (p4);
		\draw[decoration={markings,mark=at position 0.5 with {\arrow{>}}},postaction={decorate}] (p4) -- (p3);
		%		\draw[decoration={markings,mark=at position 0.5 with {\arrow{>}}},postaction={decorate}] (p3) -- (p2);
		\draw[decoration={markings,mark=at position 0.5 with {\arrow{>}}},postaction={decorate}] (p2) -- (p1);
		\draw[decoration={markings,mark=at position 0.5 with {\arrow{>}}},postaction={decorate}] (p1) -- (p0);
		
		\draw[decoration={markings,mark=at position 0.5 with {\arrow{>}}},postaction={decorate}] (q5) -- (q4);
		\draw[decoration={markings,mark=at position 0.5 with {\arrow{>}}},postaction={decorate}] (q4) -- (q3);
		%		\draw[decoration={markings,mark=at position 0.5 with {\arrow{>}}},postaction={decorate}] (q3) -- (q2);
		\draw[decoration={markings,mark=at position 0.5 with {\arrow{>}}},postaction={decorate}] (q2) -- (q1);
		\draw[decoration={markings,mark=at position 0.5 with {\arrow{>}}},postaction={decorate}] (q1) -- (q0);
		
		\draw ($(p5) + (0,0.25)$) node {$p$};
		\draw ($(q5) + (0,0.25)$) node {$q$};
		\draw ($(p0) + (0,-0.25)$) node {$r$};
		
		\draw [rotate= atan(1/5), decorate, decoration={brace,amplitude=10pt}, xshift=-4pt, yshift=0pt] (-0,0.1) -- (-0,5) node {};
		
		\draw [rotate= atan(-1/5), decorate, decoration={brace,amplitude=10pt}, xshift=-4pt, yshift=0pt] (0.25,5) -- (0.25,0.1) node {};
		
		\draw ($(-0.5,2.5) + (-0.75,-0.15)$) node {$P_1$};
		\draw ($(0.5,2.5) + (0.75,-0.15)$) node {$P_2$};
		\end{tikzpicture}
%		\subcaption{A gadget of type III}
	\end{subfigure}
	\caption{A basic gadget of type II (left) and a gadget of type III (right)}
\end{figure}
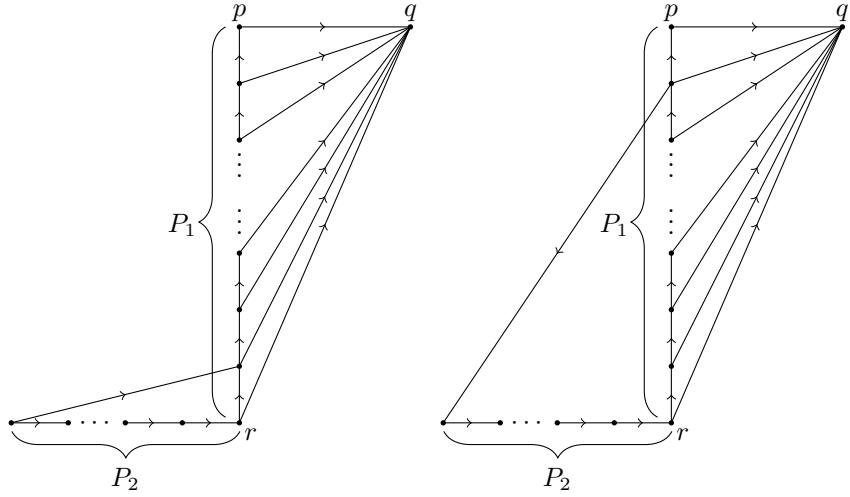
\begin{figure}\label{fig:type_II_gadget_extended}
	\definecolor{mycolor}{RGB}{0,0,0}
	\centering
	\begin{tikzpicture}[scale = 0.75]
	\foreach \i in {0,1,2,3,4,5,6,7}
	{
		\coordinate (x\i) at (0,\i);
		% \draw (x\i) node[fill=black,circle,minimum size=2pt,inner sep=0pt] {};
	}
	
	\coordinate (q) at (3,7);
	\draw (q) node[fill=mycolor,circle,minimum size=2pt,inner sep=0pt] {};
	
	\foreach \i in {0,1,2,3,5,6,7}
	{
		\draw (x\i) node[fill=mycolor,circle,minimum size=2pt,inner sep=0pt] {};
		\draw[decoration={markings,mark=at position 0.5 with {\arrow{>}}},postaction={decorate},color=mycolor] (x\i) -- (q);
	}
	
	\foreach \i in {4,5}
	{
		\draw ($(x\i) + (0,-0.3)$) node[color=mycolor] {$\vdots$};
	}
	
	\draw[decoration={markings,mark=at position 0.5 with {\arrow{>}}},postaction={decorate},color=mycolor] (x0) -- (x1);
	\draw[decoration={markings,mark=at position 0.5 with {\arrow{>}}},postaction={decorate},color=mycolor] (x1) -- (x2);
	\draw[decoration={markings,mark=at position 0.5 with {\arrow{>}}},postaction={decorate},color=mycolor] (x2) -- (x3);
	%		\draw[decoration={markings,mark=at position 0.5 with {\arrow{>}}},postaction={decorate}] (x3) -- (x4);
	\draw[decoration={markings,mark=at position 0.5 with {\arrow{>}}},postaction={decorate},color=mycolor] (x5) -- (x6);
	\draw[decoration={markings,mark=at position 0.5 with {\arrow{>}}},postaction={decorate},color=mycolor] (x6) -- (x7);
	
	\foreach \i in {1,2,3,4}
	{
		\coordinate (y\i) at (-\i,0);
		% \draw (x\i) node[fill=black,circle,minimum size=2pt,inner sep=0pt] {};
	}
	
	\foreach \i in {1,2,3,4}
	{
		\draw (y\i) node[fill=black,circle,minimum size=2pt,inner sep=0pt] {};
	}
	
	\draw ($(y3) + (0.5,0)$) node[] {$\cdots$}; 
	
	\draw[decoration={markings,mark=at position 0.5 with {\arrow{>}}},postaction={decorate}] (y4) -- (y3);
	%		\draw[decoration={markings,mark=at position 0.5 with {\arrow{>}}},postaction={decorate}] (y3) -- (y2);
	\draw[decoration={markings,mark=at position 0.5 with {\arrow{>}}},postaction={decorate}] (y2) -- (y1);
	\draw[decoration={markings,mark=at position 0.5 with {\arrow{>}}},postaction={decorate}] (y1) -- (x0);
	
	\draw[decoration={markings,mark=at position 0.5 with {\arrow{>}}},postaction={decorate}] (y4) -- (x1);
	
	\draw ($(x7) + (0,0.25)$) node[color = mycolor] {$p$};
	\draw ($(q) + (0,0.25)$) node[color = mycolor] {$q$};
	\draw ($(x0) + (0.2,-0.2)$) node[color = mycolor] {$r$};
	
	\draw [decorate,decoration={brace,amplitude=10pt},xshift=-4pt,yshift=0pt,color=mycolor]
	(-0.1,0.1) -- (-0.1,7) node {};
	\draw (-1,3.5) node[color = mycolor] {$P_1$};
	
	\draw [decorate,decoration={brace,amplitude=10pt},xshift=-4pt,yshift=0pt]
	(0.15,-0.15) -- (-3.85,-0.15) node {};
	\draw (-2,-1) node {$P_2$};
	\end{tikzpicture}
	\begin{tikzpicture}[scale = 0.75]
\foreach \i in {0,1,2,3,4,5,6,7}
{
	\coordinate (x\i) at (0,\i);
	% \draw (x\i) node[fill=black,circle,minimum size=2pt,inner sep=0pt] {};
}

\coordinate (q) at (3,7);
\draw (q) node[fill=mycolor,circle,minimum size=2pt,inner sep=0pt] {};

\foreach \i in {0,1,2,3,5,6,7}
{
	\draw (x\i) node[fill=mycolor,circle,minimum size=2pt,inner sep=0pt] {};
	\draw[decoration={markings,mark=at position 0.5 with {\arrow{>}}},postaction={decorate},color=mycolor] (x\i) -- (q);
}

\foreach \i in {4,5}
{
	\draw ($(x\i) + (0,-0.3)$) node[color=mycolor] {$\vdots$};
}

\draw[decoration={markings,mark=at position 0.5 with {\arrow{>}}},postaction={decorate},color=mycolor] (x0) -- (x1);
\draw[decoration={markings,mark=at position 0.5 with {\arrow{>}}},postaction={decorate},color=mycolor] (x1) -- (x2);
\draw[decoration={markings,mark=at position 0.5 with {\arrow{>}}},postaction={decorate},color=mycolor] (x2) -- (x3);
%		\draw[decoration={markings,mark=at position 0.5 with {\arrow{>}}},postaction={decorate}] (x3) -- (x4);
\draw[decoration={markings,mark=at position 0.5 with {\arrow{>}}},postaction={decorate},color=mycolor] (x5) -- (x6);
\draw[decoration={markings,mark=at position 0.5 with {\arrow{>}}},postaction={decorate},color=mycolor] (x6) -- (x7);

\foreach \i in {1,2,3,4}
{
	\coordinate (y\i) at (-\i,0);
	% \draw (x\i) node[fill=black,circle,minimum size=2pt,inner sep=0pt] {};
}

\foreach \i in {1,2,3,4}
{
	\draw (y\i) node[fill=black,circle,minimum size=2pt,inner sep=0pt] {};
}

\draw ($(y3) + (0.5,0)$) node[] {$\cdots$}; 

\draw[decoration={markings,mark=at position 0.5 with {\arrow{>}}},postaction={decorate}] (y4) -- (y3);
%		\draw[decoration={markings,mark=at position 0.5 with {\arrow{>}}},postaction={decorate}] (y3) -- (y2);
\draw[decoration={markings,mark=at position 0.5 with {\arrow{>}}},postaction={decorate}] (y2) -- (y1);
\draw[decoration={markings,mark=at position 0.5 with {\arrow{>}}},postaction={decorate}] (y1) -- (x0);

\coordinate (y4_aux_1) at ($(y4) + (0.25,0.75)$);
\coordinate (y4_aux_2) at ($(y4_aux_1) - (0.01,0.03)$);

%	\draw [line join=round, rounded corners = 3, decorate, decoration={zigzag, segment length=23, amplitude=2, post=lineto, post length=2pt}] (x5) -- (-2,1.5) -- (y4_aux);
%	\draw[decoration={markings,mark=at position 0.5 with {\arrow{>}}},postaction={decorate}] (y4_aux) -- (y4); 

%\draw [line join=round, rounded corners = 2, decorate, decoration={zigzag, segment length=28, amplitude=2, post=lineto, post length=2pt}] (x5) -- (-3,3) --  (y4);
%\draw[decoration={markings,mark=at position 0.5 with {\arrow{>}}},postaction={decorate}] (y4_aux_1) -- (y4_aux_2);

\draw[decoration={markings,mark=at position 0.5 with {\arrow{>}}},postaction={decorate}] (x6) -- (y4);

\draw ($(x7) + (0,0.25)$) node[color=mycolor] {$p$};
\draw ($(q) + (0,0.25)$) node[color=mycolor] {$q$};
\draw ($(x0) + (0.2,-0.2)$) node[color=mycolor] {$r$};

\draw [decorate,decoration={brace,amplitude=10pt},xshift=-4pt,yshift=0pt,color=mycolor]
(-0.1,0.1) -- (-0.1,7) node {};
\draw (-1,3.5) node[color=mycolor] {$P_1$};

\draw [decorate,decoration={brace,amplitude=10pt},xshift=-4pt,yshift=0pt]
(0.15,-0.15) -- (-3.85,-0.15) node {};
\draw (-2,-1) node {$P_2$};
\end{tikzpicture}
%	\caption{The two options for an extended gadget of type II, corresponding to the two possibilities in Item (b) in the definition of such a gadget}
	\caption{The two options for an extended gadget of type II: either there is an arc from the first vertex of $P_2$ to the second vertex of $P_1$ (left), or there is an arc from some vertex in $V(P_1) \setminus \{r\}$ to the first vertex of $P_2$ (right).}
\end{figure}

	We now introduce another useful definition. For integers $a,b \geq 1$, an 
{\em $(a,b)$-alternating-path} is an oriented path $R$ consisting of vertices $s_1,\dots,s_a,t_1,\dots,t_a$ and pairwise internally-disjoint dipaths $Q_1,\dots,Q_a,Q'_1,\dots,Q'_{a-1}$, such that $Q_i$ is a dipath from $s_i$ to $t_i$ (for each $1 \leq i \leq a)$, $Q'_i$ is a dipath from $s_{i+1}$ to $t_{i}$ (for each $1 \leq i \leq a-1$), and $Q_2,\dots,Q_{a-1},Q'_1,\dots,Q'_{a-1}$ have length at least $b$ each. We note that $Q_1$ or $Q_a$ may have length zero (in which case $s_1 = t_1$ or $s_a = t_a$, respectively). In particular, for vertices $u,v$, any dipath from $u$ to $v$ is a $(1,b)$-alternating-path with $s_1 = u$ and $t_1 = v$ (for any value of $b$); and any dipath of length at least $b$ from $u$ to $v$ is a $(2,b)$-alternating-path with $s_2 = t_2 = u$ and $s_1 = t_1 = v$. 
%??? Say something about the concatenation of alternating-paths? ??? 
The path $R$ is called {\em strong} if $Q_1$ and $Q_a$ also have length at least $b$. 
%	The {\em length} of $R$ is defined as the integer $a$. 
When several paths are considered at the same time, we will write $s_i(R),t_i(R),Q_i(R),Q'_i(R)$ (instead of $s_i,t_i,Q_i,Q'_i$) so as to prevent confusion. 
%\newline 
%??? Replace the notation for the dipaths $Q_1,\dots,Q_a,Q'_1,\dots,Q'_{a-1}$? ???
%	\begin{observation}
%		Let $a_1,a_2,b \geq 1$ be integers, and for each $i = 1,2$ let $R_i$ be a strong $(a_i,b)$-alternating path. Suppose that $t_1(R_1) = t_1(R_2)$, $s_{a_1}(R_1) = s_{a_2}(R_2)$ and that $R_1$ and $R_2$ do not share any other vertices. Then $R_1 \cup R_2$ is a subdivision of $C_{a_1+a_2-1,b}$. 
%	\end{observation}  
The following observation follows immediately from the definitions of $C_{a,b}$ and alternating-paths. 

\begin{observation}\label{obs:oriented_cycle}
	Let $a_1,a_2,b \geq 1$ be integers, and for each $i = 1,2$, let $R_i$ be a strong $(a_i,b)$-alternating path. Suppose that $s_1(R_1) = t_{a_2}(R_2)$, $s_1(R_2) = t_{a_1}(R_1)$ and that $R_1$ and $R_2$ do not share any other vertices. Then $R_1 \cup R_2$ spans a subdivision of $C_{a_1+a_2-2,b}$. 
\end{observation} 

%Let us now prove some simple facts about the I- and II--type gadgets. The following lemma 
%follows immediately from the definitions of these gadgets, and its proof is thus omitted. 

\noindent
Let us now prove some simple facts about type--I and type--II gadgets.

\begin{lemma}\label{lem:gadget_property_basic}
	Let $G$ be a gadget of type I or II (either basic or extended). Then:
	\begin{enumerate}
		\item $G$ contains a $(2,b)$-alternating path $R_0$ with $s_1(R_0)=t_1(R_0) = p$ and $t_2(R_0) = q$.
		\item $\{p,q\}$ is reachable from every vertex of $G$. 
	\end{enumerate}
\end{lemma} 
\begin{proof}
	Item 2 follows immediately from the definitions of these gadgets. Let us prove Item 1. If $G$ is of type I, i.e. a directed cycle of length at least $g > b$ through $(p,q)$, then define $R_0$ by letting $s_1(R_0) = t_1(R_0) = p$, $s_2(R_0) = t_2(R_0) = q$ and $Q'_1(R_0) = G[q,p]$ (i.e., $Q'_1(R_0)$ is simply the $q$-$p$-dipath obtained from the cycle by removing the arc $(p,q)$). If $G$ is of type II then define $R_0$ by letting $s_1(R_0) = t_1(R_0) = p$, $s_2(R_0) = r$, $t_2(R_0) = q$, $Q'_1(R_0) = P_1$ and $Q_2(R_0) = (r,q)$. 
\end{proof}

\begin{lemma}\label{lem:gadget_property_extended}
	Let $G$ be an extended gadget of type II, and let $p,q,r$ and $P_1,P_2$ be as in the definition of such a gadget. Then: 
	\begin{enumerate}
		\item For every $x \in V(G) \setminus \{p,q\}$, there exists $1 \leq a \leq 2$ and an $(a,b)$-alternating-path $R$ with $t_a(R) = x$, $s_1(R) \in \{p,q\}$ and $|V(R) \cap \{p,q\}| = 1$. 
		\item For every non-empty set $X \subseteq V(P_1) \setminus \{p,r\}$, there exists $1 \leq a \leq 2$ and an $(a,b)$-alternating-path $R$ with $t_a(R) \in X$, $s_1(R) \in \{p,q\}$ and $|V(R) \cap \{p,q\}| = |V(R) \cap X| = 1$.  
	\end{enumerate}
	
%	\newline
%	??? Add the possibility that $R$ is a $(1,b)$-alternating-path.
\end{lemma}
\begin{proof}
	We start by proving Item 2, from which Item 1 will then easily follow. So let $\emptyset \neq X \subseteq V(P_1) \setminus \{p,r\}$.
	Denote by $z$ the first vertex of $P_2$, and by $y$ the second vertex of $P_1$. 
	By the definition of an extended type-II gadget, either $(z,y) \in A(G)$ or there is some $w \in V(P_1) \setminus \{r\}$ such that $(w,z) \in A(G)$. 
	Suppose first that $(z,y) \in A(G)$. Traverse the dipath $P_1$ starting from $y$ until the first vertex of $X$ is reached, and denote this vertex by $x$. Evidently, we have $V(G) \cap V(P_1[y,x]) = \{x\}$.   
	Now define $R$ by setting $s_1(R) = t_1(R) = q$, $s_2(R) = z$, $t_2(R) = x$, $Q'_1(R) = P_2 \circ (r,q)$ and $Q_2(R) = (z,y) \circ P_1[y,x]$. Then $R$ is indeed a $(2,b)$-alternating-path (since $|P_2| \geq b$), and we have $V(R) \cap \{p,q\} = \{q\}$ (since $x \neq p$) and $V(R) \cap X = \{x\}$, as required. 
	
	Suppose now that there is $w \in V(P_1) \setminus \{r\}$ such that $(w,z) \in A(G)$.
	If $w = p$ then, as before, we let $x$ be the first vertex of $X$ reached when traversing $P_1[y,p]$. Observe that $(w,z) \circ P_2 \circ P_1[r,x]$ is a dipath from $p = w$ to $x$, and thus also a $(1,b)$-alternating-path $R$ with $s_1(R) = p$ and $t_1(R) = x$. Moreover, our choice of $x$ implies that $V(R) \cap X = \{x\}$, as required. 
	
	So from now we assume that $w \neq p$. In this case, choose an element $x' \in X$, which is closest to $w$ in the {\em undirected} path underlying $P_1$. In other words, we choose $x' \in X$ such that the subpath of $P_1$ between $w$ and $x'$ contains no vertex of $X$ other than $x'$ itself.  
	We consider two cases, based on the relative position of $x'$ and $w$ along $P_1$. 
%	Assume first either that $w = x'$, or that when traversing $P_1$ (starting from $r$), $w$ is reached before $x'$ is. 
	Assume first that when traversing the dipath $P_1$ (starting from $r$), $w$ is reached before $x'$ is (here we allow $w = x'$). 
	In this case, define $R$ by setting $s_1(R) = t_1(R) = q$, $s_2(R) = w$, $t_2(R) = x'$, $Q'_1(R) = (w,z) \circ P_2 \circ (r,q)$ and $Q_2(R) = P_1[w,x']$. 
	Assume now that $x'$ is reached before $w$ when traversing $P_1$. 
	In this case, define a $(2,b)$-alternating-path $R$ by setting 
	$s_1(R) = t_1(R) = q$, $s_2(R) = t_2(R) = x'$ and
	$Q'_1(R) = P_1[x',w] \circ (w,z) \circ P_2 \circ (r,q)$. 
	Observe that in both cases, $R$ is indeed a $(2,b)$-alternating-path (because $|P_2| \geq b$), $V(R) \cap \{p,q\} = \{q\}$ (because $w,x' \neq p$), and $V(R) \cap X = \{x'\}$ (by our choice of $x'$). This concludes the proof of Item 2. 
	
	It remains to prove Item 1. So let $x \in V(G) \setminus \{p,q\}$. 
	If $x \in V(P_2)$, then we define a $(2,b)$-alternating-path $R$ by setting $s_1(R) = t_1(R) = p$, $s_2(R) = t_2(R) = x$ and $Q'_1(R) = P_2[x,r] \circ P_1$. And if 
	$x \in V(G) \setminus (V(P_2) \cup \nolinebreak \{p,q\}) = V(P_1) \setminus \{p,r\}$, then we obtain the required alternating-path $R$ by applying Item 2 with $X := \{x\}$. 
	It is easy to see that in both cases, $R$ satisfies the assertion of Item 1. This completes the proof of the lemma. 
\end{proof} 

\subsection{Gadget Chains}
We now define the notion of a {\em chain} of gadgets, a structure which will be instrumental to our proof of Theorem \ref{thm:cycle_orientation_main}. In what follows, for a gadget $G$, we will denote by $p(G)$ and $q(G)$ the designated vertices $p$ and $q$ of $G$, respectively. 
%By ``gadget" we will mean a gadget of one of the four types defined above. 
\begin{definition}\label{def:chain}
	A {\em chain} $\mathcal{C}$ consists of a directed path $P = v_0,\dots,v_m$, a partition 
	$A_1 \cup A_2 = A(P) = \{(v_0,v_1),\dots,(v_{m-1},v_m)\}$ of the arc-set of $P$, and a collection of (non-trivial) gadgets $(G_e : e \in A_2)$ having the following four properties:
	\begin{enumerate}
		\item For every $e \in A_2$, the gadget $G_e$ is either of type I, type III, or basic type-II.
		\item For every $e = (v_i,v_{i+1}) \in A_2$, $p(G_e) = v_i$ and $q(G_e) = v_{i+1}$. 
%		\item The sets $V(G_{(v_i,v_{i+1})}) \setminus \{v_i,v_{i+1}\}$, where $(v_i,v_{i+1}) \in A_2$, are pairwise-disjoint. Moreover, each of these sets is disjoint from $\{v_0,\dots,v_m\}$.  
%		\item  
%		The sets 
%		$\left( V(G_{(v_i,v_{i+1})}) \setminus \{v_i,v_{i+1}\} \; : \; (v_i,v_{i+1}) \in A_2 \right)$ are pairwise-disjoint, and
%%		each of these sets is disjoint from $\{v_0,\dots,v_m\}$. 
%		\linebreak$V(G_{(v_i,v_{i+1})}) \cap \{v_0,\dots,v_m\} = \{v_i,v_{i+1}\}$ for every $(v_i,v_{i+1}) \in A_2$, 
%		and 
%		$V(G_e) \cap V(G_f) \subseteq \{v_0,\dots,v_m\}$ for every pair of distinct $e,f \in A_2$.
		\item $V(G_{(v_i,v_{i+1})}) \cap \{v_0,\dots,v_m\} = \{v_i,v_{i+1}\}$ for every $(v_i,v_{i+1}) \in A_2$.
		\item $V(G_e) \cap V(G_f) \subseteq \{v_0,\dots,v_m\}$ for every pair of distinct $e,f \in A_2$.
	\end{enumerate} 
	We will use the following terminology and notation:
	\begin{itemize}
	\item With a slight abuse of notation, we identify the chain $\mathcal{C}$ and the digraph consisting of the union of $P$ and the gadgets $G_e, e \in A_2$.
		\item For convenience, for $(v_i,v_{i+1}) \in A_1$ we denote by $G_{(v_i,v_{i+1})}$ the trivial gadget with vertices $v_i, v_{i+1}$ and arc $(v_i,v_{i+1})$. 
		\item In cases where several chains are considered at the same time, we will write $A_1(\mathcal{C})$, $A_2(\mathcal{C})$ and $G_e(\mathcal{C})$ to indicate that we are considering the chain $\mathcal{C}$. 
		\item The dipath $P$ is called the {\em spine} of the chain, and $|P| = m$ is the {\em length} of the chain.
		\item The vertex set of $\mathcal{C}$, denoted $V(\mathcal{C})$, is defined as $V(\mathcal{C}) = V(P) \cup \bigcup_{e \in A_2}V(G_e)$.
%		\item For each vertex $x \in V(\mathcal{C})$, define $d_{\mathcal{C}}(x)$ to be the shortest length of a dipath in $\mathcal{C}$ from $x$ to $v_t$. Note that $d(v_i) = m - i$ for every $0 \leq i \leq m$, and that for every $e = (v_i,v_{i+1}) \in A_2$ and $x \in V(G_e)$ it holds that $d(x) > d(v_{i+1})$. 
		\item For integers $0 \leq i < j \leq m$, we denote by $\mathcal{C}[v_i,v_j]$ the {\em subchain} of $\mathcal{C}$ whose spine is $P[v_i,v_j] = v_i,v_{i+1}, \dots,v_j$; so $A_{\ell}(\mathcal{C}[v_i,v_j]) = A_{\ell}(\mathcal{C}) \cap A(P[v_i,v_j])$ for $\ell = 1,2$, and $\mathcal{C}[v_i,v_j]$ inherits the gadgets of $\mathcal{C}$. 
%		\item We say that $\mathcal{C}$ is {\em tight} if $(v_{m-1},v_m) \in A_2$ and if among any $\dots$ consecutive arcs of $P$, there is an arc belonging to $A_2$.   
	\end{itemize}
%	The {\em length} of $\mathcal{C}$ is defined as $t = |A(P)|$. 
%	The dipath $P$ is called the {\em spine} of the chain. 
%%	The {\em arc-set} of the chain is $A(P) \cup \{(w_e,v_i),(w_e,v_{i+1}) : e = (v_i,v_{i+1}) \in A_3, \; 0 \leq i \leq t-1\}$.
%	The vertex set of $\mathcal{C}$, denoted $V(\mathcal{C})$, is defined as the union of the sets $V(P) = \{v_0,\dots,v_t\} \cup \bigcup_{e \in A_2}V(G_e)$. For integers $0 \leq i < j \leq t$, we denote by $\mathcal{C}[v_i,v_j]$ the {\em subchain} of $\mathcal{C}$ with spine $P[v_i,v_j] = v_i,v_{i+1}\dots,v_j$ (and with the same gadgets as $\mathcal{C}$).  
%	We will say that $\mathcal{C}$ is {\em tight} if $(v_{t-1},v_t) \in A_2$ and if among any $\dots$ consecutive arcs of $P$, there is an arc belonging to $A_2$. 
\end{definition} 
%	In other words, Item 2 states that every $e \in A_2$ spans a digon, and Item 3 states that for every $e \in A_3$, both endpoints of $e$ are dominated by $w_e$. Tightness means that $(v_{t-1},v_t) \in A_2 \cup A_3$ and that $A_1$ cannot contain a pair of consecutive edges of $P$. 
%	The following lemma demonstrates that chains are useful for finding subdivisions of (acyclic) oriented cycles. 
The next sequence of lemmas is concerned with embedding subdivisions of $C_{a,b}$ using gadget chains. The following lemma asserts that such chains can be used to find $(a,b)$-alternating-paths. 
%Let $P_{a,b}$ denote the oriented path obtained from $C_{a,b}$ by deleting the dipath from $s_a$ to $t_1$. In other words, $P_{a,b}$ consists of $2a$ vertices $s_1,\dots,s_a,t_1,\dots,t_a$ and of $2a-1$ internally-disjoint length-$b$ dipaths: one from $s_i$ to $t_i$ for each $1 \leq i \leq a$, and one from $s_i$ to $t_{i+1}$ for each $1 \leq i \leq a-1$. 
	\begin{lemma}\label{lem:chain_anti_dir_path}
		Let $a,b \geq 1$ be integers. Let $\mathcal{C}$ be a chain, let $P = v_0,\dots,v_m$ and $A_1,A_2$ be as in Definition \ref{def:chain}, and suppose that $|A_2| \geq a(b+1) - 1$. Then
		$\mathcal{C}$ contains a strong $(a,b)$-alternating-path $R$ 
		with $s_1(R) = v_0$ and $t_a(R) = v_m$. 
	\end{lemma}
	\begin{proof}
		The proof is by induction on $a$. 
		In the base case $a = 1$, the condition in the lemma states that $|A_2| \geq b$. This implies that $m = |A(P)| \geq b$, meaning that $P$ is a dipath of length at least $b$ from $v_0$ to $v_m$, and hence also a strong $(1,b)$-alternating-path with $s_1(P) = v_0$ and $t_1(P) = v_m$. 
		
		We now move on to the induction step. So let $a \geq 2$. 
		Let $j$ be the largest integer in the set $\{0,\dots,m-b-1\}$ satisfying $(v_j,v_{j+1}) \in A_2$. Set $\mathcal{C}' := \mathcal{C}[v_0,v_j]$. Then 
		$|A_2(\mathcal{C}')| \geq |A_2(\mathcal{C})| - (b+1) \geq (a-1)(b+1) - 1$. By the induction hypothesis, $\mathcal{C}'$ contains a strong $(a-1,b)$-alternating-path $R'$ with $s_1(R') = v_0$ and $t_{a-1}(R') = v_j$.   
		
		Setting $e := (v_j,v_{j+1})$, suppose first that $G_e$ is either of type I or a basic gadget of type II. By Item 1 of Lemma \ref{lem:gadget_property_basic}, $G_e$ contains a $(2,b)$-alternating path $R_0$ with $s_1(R_0)=t_1(R_0) = v_j$ and $t_2(R_0) = v_{j+1}$.
		Now let $R$ be the $(a,b)$-alternating-path obtained by attaching to $R'$ the dipaths $Q'_1(R_0)$ and $Q_2(R_0) \circ P[v_{j+1},v_m]$. Formally, $R$ is defined by setting 
		$s_i(R) = s_i(R')$ and $t_i(R) = t_i(R')$ for every $1 \leq i \leq a-1$ (so in particular, $s_1(R) = v_0$), $s_a(R) = s_2(R_0)$, $t_a(R) = v_m$, $Q_i(R) = Q_i(R')$ for every $1 \leq i \leq a-1$, $Q'_i(R) = Q'_i(R')$ for every $1 \leq i \leq a-2$, $Q'_{a-1}(R) = Q'_1(R_0)$ and 
		$Q_a(R) = Q_2(R_0) \circ P[v_{j+1},v_m]$. Note that $|Q_a(R)| \geq b$ because $j \leq m-b-1$. 
		It follows that $R$ is indeed a strong $(a,b)$-alternating-path, as required.
		
		Suppose now that $G_e$ is of type III. Then $G_e$ consists of the vertices $v_j,v_{j+1}$, a vertex $r$, and two internally vertex-disjoint dipaths $P_1,P_2$ from $v_j$ and $v_{j+1}$, respectively, to $r$, such that $P_1$ and $P_2$ have length at least $2b-1 \geq b$ each. 
		Now let $R$ be the $(a,b)$-alternating-path obtained by attaching to $R'$ the dipaths 
%		$P_1$ and $P_2 \circ P[v_{j+1},v_m]$. 
		$P_1$, $P_2$ and $P[v_{j+1},v_m]$. 
		Formally, $R$ is defined by setting 
		$s_i(R) = s_i(R')$ for every $1 \leq i \leq a-1$ (so in particular, $s_1(R) = v_0$), $t_i(R) = t_i(R')$ for every $1 \leq i \leq a-2$, $t_{a-1}(R) = r$, $s_a(R) = v_{j+1}$, $t_a(R) = v_m$, $Q_i(R) = Q_i(R')$ and $Q'_i(R) = Q'_i(R')$ for every $1 \leq i \leq a-2$, $Q_{a-1}(R) = Q_{a-1}(R') \circ P_1$, $Q'_{a-1}(R) = P_2$ and 
		$Q_a(R) = P[v_{j+1},v_m]$.  
		Again, it is easy to check that $R$ is a strong $(a,b)$-alternating-path, as required. 
	\end{proof}
%	??? Define concatenation of dipaths (including the case where a dipath is concatenated with an arc). Also, define the notation $|P|$ for the length of a path $P$ and use it in the proof of Lemma \ref{lem:gadget_intersection} (and elsewhere if needed). ??? 

%	\begin{lemma}
%		Let $G,G^*$ be gadgets such that 
%		$V(G) \cap V(G^*) \neq \emptyset$
%%		and $\{p(G),q(G)\} \cap \{p(G^*),q(G^*)\} = \emptyset$, 
%		and such that $G^*$ is of type I,II or III. Then there exists $1 \leq a \leq 3$ such that $G \cup G^*$ contains an $(a,b)$-alternating-path $R$ satisfying the following properties:
%		\begin{enumerate}
%			\item $t_1(R) \in \{p(G),q(G)\}$ and $s_a(R) \in \{p(G^*),q(G^*)\}$.
%			\item $|V(R) \cap \{ p(G),q(G) \}| = 1$ and $|V(R) \cap \{p(G^*),q(G^*)\}| = 1$. 
%		\end{enumerate}
%%		??? Can we really guarantee that $p^* \notin V(R)$ in all cases? In (the first part of) Item 3 of Lemma \ref{lem:gadget_property}, should we add the requirement that the path from $q$ to $x$ avoids $p$? Should we also require that if $t_1(R) = p$ then $q \notin V(R)$? ???. 
%%		\newline 
%%		??? Assume that $\{p(G),q(G)\} \cap \{p(G^*),q(G^*)\} = \emptyset$? Is it enough to require that $p^* \notin V(G)$? ּּ???
%	\end{lemma}

	\begin{lemma}\label{lem:gadget_intersection}
	Let $G,G^*$ be gadgets such that $V(G) \cap V(G^*) \neq \emptyset$, 
$p(G^*),q(G^*) \notin V(G)$, \linebreak and $G^*$ is 
%	either of type I or an extended gadget of type II. 
an extended gadget of type II. Then there exists $1 \leq a \leq 3$ such that $G \cup G^*$ contains an $(a,b)$-alternating-path $R$ with $t_a(R) \in \{p(G),q(G)\}$,
$s_1(R) \in \{p(G^*),q(G^*)\}$ and
$|V(R) \cap \nolinebreak \{ p(G),q(G) \}| = |V(R) \cap \{p(G^*),q(G^*)\}| = 1$. 
	%		??? Can we really guarantee that $p^* \notin V(R)$ in all cases? In (the first part of) Item 3 of Lemma \ref{lem:gadget_property}, should we add the requirement that the path from $q$ to $x$ avoids $p$? Should we also require that if $t_1(R) = p$ then $q \notin V(R)$? ???. 
	%		\newline 
	%		??? Assume that $\{p(G),q(G)\} \cap \{p(G^*),q(G^*)\} = \emptyset$? Is it enough to require that $p^* \notin V(G)$? ּּ???
	\end{lemma}
%	??? Change the notation used for the various paths in the proof of Lemma \ref{lem:gadget_intersection}? ???
	\begin{proof}
	For convenience, let us put $p := p(G)$, $q := q(G)$, $p^* := p(G^*)$ and $q^* := q(G^*)$. The assumption $p^*,q^* \notin V(G)$ will be used implicitly throughout the proof. 
%	Note that if $\{p,q\} \cap \{p^*,q^*\}$ then the assertion of the lemma holds trivially (as one can simply take $R$ to be an empty path from a vertex in $\{p,q\} \cap \{p^*,q^*\}$ to itself), so we will (implicitly) assume throughout the proof that $\{p,q\} \cap \{p^*,q^*\} = \emptyset$. 
	We proceed by a case analysis over the types of $G$ and $G^*$. 
	\paragraph{Case 1.} $G$ is trivial, a gadget of type I, or a gadget of type II. 
	Recall that $V(G) \cap V(G^*) \neq \emptyset$ by assumption. 
	By Item 2 of Lemma \ref{lem:gadget_property_basic}, $\{p,q\}$ is reachable from every vertex of $V(G)$ via a dipath inside $G$ (this is evident if $G$ is trivial). In particular, $G$ contains a dipath from $V(G) \cap V(G^*)$ to $\{p,q\}$. 
%	Fix a shortest dipath $P$ in $G$ which starts in $V(G) \cap V(G^*)$ and ends in $\{p,q\}$. Let $x \in V(G) \cap V(G^*)$ be the first vertex of $P$. 
	Fix a shortest such dipath $P \subseteq G$, and let $x \in V(G) \cap V(G^*)$ be the first vertex of $P$. 
	The minimality of $P$ implies that $V(P) \cap V(G^*) = \{x\}$ and $|V(P) \cap \{p,q\}| = 1$. 
%	Fix a vertex $x \in V(G) \cap V(G^*)$, whose distance to $\{p,q\}$ inside $G$ is smallest among all vertices of $V(G) \cap V(G^*)$. Let $P$ be a shortest dipath from $x$ to $\{p,q\}$ in $G$. Observe that by our choice of $x$, we have $V(P) \cap V(G^*) = \{x\}$ (since otherwise there would be a vertex in $V(P) \cap V(G^*) \subseteq V(G) \cap V(G^*)$ whose distance to $\{p,q\}$ inside $G$ is smaller than that of $x$). The minimality of $P$ also implies that $|V(P) \cap \{p,q\}| = 1$. 
%	Suppose first that $G^*$ is of type I, i.e. a directed cycle. Then $G^*$ contains a dipath $P^*$ from $\{p^*,q^*\}$ to $x$ with $|V(P^*) \cap \{p^*,q^*\}| = 1$. Note that $V(P) \cap V(P^*) = \{x\}$ because $V(P^*) \subseteq V(G^*)$ and $V(P) \cap V(G^*) = \{x\}$. It follows that $R := P^* \circ P$ is a dipath from $\{p^*,q^*\}$ to $\{p,q\}$ satisfying $|V(R) \cap \{ p,q \}| = |V(R) \cap \{p^*,q^*\}| = 1$. In particular, $R$ is a $(1,b)$-alternating-path satisfying the requirements in the lemma. 
%	Suppose now that $G^*$ is an extended gadget of type II.
	By Item 1 of Lemma \ref{lem:gadget_property_extended}, there is $1 \leq a \leq 2$ such that $G^*$ contains an $(a,b)$-alternating-path $R^*$ with $t_a(R^*) = x$, $s_1(R^*) \in \{p^*,q^*\}$ and $|V(R^*) \cap \{p^*,q^*\}| = 1$. (The condition $x \notin \{p^*,q^*\}$ appearing in Item 1 of Lemma \ref{lem:gadget_property_extended} is satisfied here because $x \in V(G) \cap V(G^*)$ whereas $p^*,q^* \notin V(G)$ by assumption.)
	Note that $V(P) \cap V(R^*) = \{x\}$ because 
	$V(P) \cap V(G^*) = \{x\}$ and $V(R^*) \subseteq V(G^*)$. 
	Now it is easy to see that by combining $P$ and $R^*$ we obtain an $(a,b)$-alternating-path $R$ with $t_a(R) \in \{p,q\}$, $s_1(R) \in \{p^*,q^*\}$ and 
	$|V(R) \cap \{ p,q \}| = |V(R) \cap \{p^*,q^*\}| = 1$. Formally, $R$ is defined by setting $Q_a(R) := Q_a(R^*) \circ P$ (so $t_a(R) \in \{p,q\}$ is the last vertex of $P$); $s_i(R) = s_i(R^*)$ for every $1 \leq i \leq a$; and $t_i(R) = t_i(R^*)$, $Q_i(R) = Q_i(R^*)$ and $Q'_i(R) = Q'_i(R^*)$ for every $1 \leq i \leq a-1$. This completes the proof in Case 1.      
	
	\paragraph{Case 2.} $G$ is a gadget of type III. In this case $G$ consists of the arc $(p,q)$, a vertex $r$, and two internally vertex-disjoint dipaths $P_1,P_2$ from $p$ and $q$, respectively, to $r$, such that $P_1$ and $P_2$ have length at least $2b-1$ each. 
	As $G^*$ is an extended gadget of type II, it consists of the vertices $p^*,q^*$, a vertex $r^*$ and dipaths $P^*_1,P^*_2$, all satisfying the properties stated in the definition of a type-II gadget. 
%	We start by handling the case that there is some $x \in V(P^*_1)$ such that $x \in V(G)$ and the distance from $\{p,q\}$ to $x$ in $G$ is at least $b-1$. 
	We start by handling the case that there is some $x \in V(P^*_1) \cap V(G)$ such that the distance from $\{p,q\}$ to $x$ in $G$ is at least $b-1$. 
	Since $V(G) = V(P_1) \cup V(P_2)$, we have either $x \in V(P_1)$ or $x \in V(P_2)$. 
	Suppose without loss of generality that $x \in V(P_1)$ (the case $x \in V(P_2)$ is symmetric). Our assumption on $x$ then means that $|P_1[p,x]| \geq b-1$. 
%	Note that every vertex of $G$ is reachable from $\{p,q\}$ (via either $P_1$ or $P_2$), so fix a shortest dipath $P \subseteq G$ from $\{p,q\}$ to $x$. Then $|V(P) \cap \{p,q\}| = 1$ and $P$ has length at least $b-1$. 
%	Now, by concatenating the dipath $P$ with the arc $(x,q^*) \in A(G^*)$ (which exists by the definition of a type-II gadget), we obtain a dipath $R$ of length at least $b$ from $\{p,q\}$ to $q^*$. 
	Now, $R := P_1[p,x] \circ (x,q^*)$ is a dipath of length at least $b$ from $\{p,q\}$ to $q^*$ (note that $(x,q^*) \in A(G^*)$ by the definition of a type-II gadget).
	Hence, $R$ constitutes a $(2,b)$-alternating-path with $s_1(R) = t_1(R) = q^*$ and $s_2(R) = t_2(R) \in \{p,q\}$.  
	Moreover, $|V(R) \cap \{p,q\}| = 1$ and $V(R) \cap \{p^*,q^*\} = \{q^*\}$ (since $p^* \notin V(G)$), as required.  
	 
	So from now on we assume that every $x \in V(P^*_1) \cap V(G)$ is at distance at most $b-2$ from $\{p,q\}$ in $G$ (in particular, if $b = 1$ then $V(P^*_1) \cap V(G) = \emptyset$). 
	It follows that 
	$|V(G) \cap V(P^*_1)| \leq 2(b-1)$. 
%	\newline 
%	??? What happens if $b-2 < 0$, i.e. if $b = 1$? ???
%	Moving forward, we will consider the intersection of $V(G)$ with $V(P^*_1) \setminus \{r^*\}$, starting with the case $V(G) \cap \left( V(P^*_1) \setminus \{r^*\} \right) = \emptyset$. 
%	Since $V(G) \cap V(G^*) \neq \emptyset$ and $p^*,q^* \notin V(G)$, we must have that $V(G) \cap V(P_2^*) \neq \emptyset$. Now, traverse the dipath $P_2^*$ {\em backwards} (starting from its last vertex, $r^*$) until a first vertex of $V(G)$ is reached, and denote this vertex by $y$. Evidently $V(G) \cap P^*_2[y,r] = \{y\}$.  
%	Moving forward, we will consider the intersection of $V(G)$ with $V(P^*_2)$, starting with the case $V(G) \cap V(P^*_2) = \emptyset$. 
	Moving forward, we will consider two cases, based on the intersection of $V(G)$ with $V(P^*_2)$. 
	\paragraph{Case 2.1.} $V(G) \cap V(P^*_2) = \emptyset$. Set $X := V(G) \cap V(G^*)$, noting that $X \neq \emptyset$ by assumption.
	As $V(G) \cap V(P^*_2) = \emptyset$ and $p^*,q^* \notin V(G)$, we must have that 
	$X \subseteq V(G^*) \setminus (V(P^*_2) \cup \{p^*,q^*\}) = V(P^*_1) \setminus \{p^*,r^*\}$. By Item 2 of Lemma \ref{lem:gadget_property_extended}, there exists $1 \leq a \leq 2$ and an $(a,b)$-alternating-path $R^*$ contained in $G^\ast$, such that $t_a(R^*) \in X$, $s_1(R^*) \in \{p^*,q^*\}$ and $|V(R^*) \cap \{p^*,q^*\}| = |V(R^*) \cap X| = 1$. 
	For convenience, put $x := t_a(R^*)$. Note that $V(R^*) \cap V(G) = \{x\}$ by our choice of $X$ and $R^*$. We now see that if $x \in \{p,q\}$, then $R := R^*$ satisfies the requirements of the lemma. Suppose then that $x \notin \{p,q\}$.   
	Since $x \in V(G)$, we have either $x \in V(P_1)$ or $x \in V(P_2)$. Without loss of generality, we assume that $x \in V(P_1)$ (the case that $x \in V(P_2)$ is symmetric). Recall that by our assumption, $x$ is at distance at most $b-2$ from $\{p,q\}$ in $G$; in other words, 
	the length of the dipath $P_1[p,x]$ is at most $b-2$. As $|P_1| \geq 2b-1 \geq 2b-2$, we get that $|P_1[x,r]| = |P_1| - |P_1[p,x]| \geq b$. Now let $R$ be the $(a+1,b)$-alternating-path obtained by combining $R^*$ with the dipaths $P_1[x,r]$ and $P_2$. Formally, $R$ is defined by setting $s_i(R) = s_i(R^*)$ for every $1 \leq i \leq a$; $t_i(R) = t_i(R^*)$,  $Q_i(R) = Q_i(R^*)$ and $Q'_i(R) = Q'_i(R^*)$ for every $1 \leq \nolinebreak i \leq \nolinebreak a-1$; $t_a(R) = r$; $s_{a+1}(R) = t_{a+1}(R) = q$; $Q_a(R) = Q_a(R^*) \circ P_1[x,r]$; and $Q'_a(R) = P_2$. Note that $R$ is indeed an $(a+1,b)$-alternating-path; this follows from our choice of $R^*$, the fact that $V(R^*) \cap V(G) = \{x\}$, and the bounds $|Q'_{a}(R)| = |P_2| \geq 2b-1 \geq b$ and $|Q_a(R)| \geq |P_1[x,r]| \geq b$. 
	We also have $|V(R) \cap \{p^*,q^*\}| = 1$ (as $V(R) \cap \{p^*,q^*\} = V(R^*) \cap \{p^*,q^*\}$) and $V(R) \cap \{p,q\} = \{q\}$ (by our definition of $R$ and as $x \notin \{p,q\}$).
	Since $a+1 \le 3$, we see that the assertion of the lemma holds in 
%	the case $V(G) \cap V(P^*_2) = \emptyset$. 
	Case 2.1.
%	It remains to handle the case $V(G) \cap V(P^*_2) \neq \emptyset$. 
	
	\paragraph{Case 2.2.} $V(G) \cap V(P^*_2) \neq \emptyset$. In this case, we traverse the dipath $P^*_2$ {\em backwards} (i.e., starting from its last vertex, $r^*$), until the first time a vertex of $V(G)$ is reached, and denote this vertex by $w$. Evidently, $V(G) \cap V(P^*_2[w,r^*]) = \{w\}$. For convenience, let us set $P^* := P_2^*[w,r^*] \circ P^*_1$, noting that $P^*$ starts at $w$, ends at $p^*$, and has length at least $|P^*_1| \geq 2b^2 + b - 2$ (by the definition of a type-II gadget). For every $u \in V(G) \cap (V(P^*) \setminus \{w\})$, denote by $e_u$ the (unique) arc of $P^*$ whose head is $u$. Let $R_1,\dots,R_m$ be the connected components of the digraph obtained from $P^*$ by deleting the arc $e_u$ for every $u \in V(G) \cap (V(P^*) \setminus \{w\})$ (this digraph is a dipath forest). Then for each $1 \leq i \leq m$, $R_i$ is a dipath whose first vertex is in $V(G)$ and all of whose other vertices are not in $V(G)$. Recall that by our assumption, $|V(G) \cap V(P^*_1)| \leq 2(b-1)$. Now, our choice of $w$ implies that $V(G) \cap V(P^*) = (V(G) \cap V(P^*_1)) \cup \{w\}$. 
%	Hence, 
%	$|V(G) \cap V(P^*) | \leq 2(b-1) + 1 = 2b-1$. Furthermore, we have $m \leq 2b-1$, since the number of arcs we deleted from $P^*$ (to obtain $R_1,\dots,R_m$) is at most $|V(G) \cap V(P^*_1)| \leq 2(b-1)$. 
	The number of edges we deleted from $P^*$ to obtain $R_1,\dots,R_m$ is, one the one hand, equal to $m-1$, and on the other hand equal to $|V(G) \cap (V(P^*) \setminus \{w\})| \leq |V(G) \cap V(P^*_1)| \leq 2(b-1)$.  
	It follows that $m \leq 2(b-1)+1 = 2b-1$ and 
	$|R_1| + \dots + |R_m| \geq |P^*| - 2(b-1) \geq 2b^2 + b - 2 - 2(b-1) = 2b^2 - b$. By averaging, there is some $1 \leq i \leq m$ such that $|R_i| \geq \frac{2b^2 - b}{m} \geq 
	\frac{2b^2 - b}{2b-1} \geq b$. 
	
	Let $u$ (resp. $v$) be the first (resp. last) vertex of $R_i$. Note that $v \in V(P^*_1)$ due to our choice of $w$. Define a dipath $R^*$ as follows: if $v = p^*$ then set $R^* := R_i$, and otherwise set $R^* := R_i \circ (v,q^*)$. (That $(v,q^*) \in A(G^*)$ follows from the definition of a type-II gadget and the fact that $v \in V(P^*_1)$.) Then $|V(R^*) \cap \{p^*,q^*\}| = 1$ and $V(G) \cap V(R^*) = \{u\}$ (because $V(G) \cap V(R_i) = \{u\}$ and $q^* \notin V(G)$). 
	In particular, 
	$u \in V(G) = V(P_1) \cup V(P_2)$. 
%	As mentioned above, we have \linebreak$(V(R_i) \setminus \{u\}) \cap V(G) = \emptyset$ and $u \in V(G) = V(P_1) \cup V(P_2)$.  
	Suppose, without loss of generality, that $u \in V(P_1)$ (the case $u \in V(P_2)$ is symmetric). 
	Now define a $(2,b)$-alternating-path $R$ as follows: 
	$s_1(R) = t_1(R) = q^*$, 
	$s_2(R) = t_2(R) = p$,  
	$Q'_1(R) = P_1[p,u] \circ R^*$. 
	Note that $Q'_1(R)$ is indeed a dipath (because $V(P_1) \cap V(R^*) \subseteq V(G) \cap V(R^*) = \{u\}$), and that $|Q'_1(R)| \geq |V(R^*)| \geq |V(R_i)| \geq b$.
%	It is also easy to see that $q,p^* \notin V(R)$. 
	Furthermore, $q \notin V(R)$ and $V(R) \cap \{p^*,q^*\} = V(R^*) \cap \{p^*,q^*\}$.  
	Thus, $R$ satisfies the requirements of the lemma. This completes the proof.  
\end{proof}

%We note that Lemma \ref{lem:gadget_intersection} is also true without the condition that $p(G^*),q(G^*) \notin V(G)$, although this requires a slightly longer case analysis. 
\noindent
Note that the gadget $G$ in Lemma \ref{lem:gadget_intersection} is allowed to be trivial.

In the following lemma we show that if a sufficiently ``rich" chain (i.e., a chain $\mathcal{C}$ with $|A_2(\mathcal{C})|$ large enough) ``self-intersects" in some well-defined way, then it contains a subdivision of $C_{a,b}$.
Roughly speaking, this can be thought of as closing the alternating-path obtained from Lemma \ref{lem:chain_anti_dir_path} to form a cycle (i.e., a $C_{a,b}$-subdivision). 
%In some cases, this closure is achieved by using Lemma \ref{lem:gadget_intersection}. 
The purpose of Lemma \ref{lem:gadget_intersection} is to achieve this closure. 
\begin{lemma}\label{lem:chain_main}
	Let $a \geq 2$ and $b \geq 1$ be integers, let $\mathcal{C}$ be a chain contained in a digraph $D$ and let $P = z_0,\dots,z_{\ell}$ and $A_1,A_2$ be as in Definition \ref{def:chain}. Suppose that $|A_2| \geq (a+3)(b+1) - 2$ and that at least one of the following two conditions is satisfied:
	\begin{enumerate}
%		\item $(v_m,v_0) \in A(D)$;
%		\item $(v_0,v_1) \in A_2$ and there is $x \in G_{(v_0,v_1)}$ such that $(v_m,x) \in A(D)$. 
		\item There exists $x \in V(G_{(z_0,z_1)})$ such that $(z_{\ell},x) \in A(D)$. 
%		\item There exists a vertex $z^* \in V(D) \setminus V(\mathcal{C})$ such that $(z_{\ell},z^*) \in A(D)$, and there exists a type-I or extended type-II gadget $G^*$ such that $p(G^*) = z_{\ell}$, $q(G^*) = z^*$,  $V(G_{(z_0,z_1)}) \cap \nolinebreak V(G^*) \neq \emptyset$ and 
%		$V(\mathcal{C}) \cap V(G^*) \subseteq V(G_{(z_0,z_1)}) \cup \{z_{\ell}\}$. 
		\item There exists a vertex $z^* \in V(D) \setminus V(\mathcal{C})$ such that $(z_{\ell},z^*) \in A(D)$, and there exists an extended type-II gadget $G^*$ such that $p(G^*) = z_{\ell}$, $q(G^*) = z^*$,  $V(G_{(z_0,z_1)}) \cap \nolinebreak V(G^*) \neq \emptyset$ and 
		$V(\mathcal{C}) \cap V(G^*) \subseteq V(G_{(z_0,z_1)}) \cup \{z_{\ell}\}$. 
	\end{enumerate}
	Then $D$ contains a subdivision of $C_{a,b}$.   
\end{lemma}
\begin{proof}
	%		Let $\mathcal{C}' := \mathcal{C}[v_0,v_{m-b}]$ be the subchain of $\mathcal{C}$ with spine $P[v_0,v_{t-b}]$. 
	For convenience, put $G := G_{(z_0,z_1)}$. 
	We start by showing that for some $1 \leq a_1 \leq 3$, $G \cup G^*$ contains an $(a_1,b)$-alternating-path $R^*$ satisfying
	$t_{a_1}(R^*) \in \{z_0,z_1\}$,
	$s_1(R^*) \in \{z_{\ell},z^*\}$ and 
	$|V(R^*) \cap \{ z_0,z_1 \}| = |V(R^*) \cap \{z_{\ell},z^*\}| = 1$. 
	If Condition 2 in the lemma holds, then this assertion follows immediately from Lemma \ref{lem:gadget_intersection}. Note that the conditions of Lemma \ref{lem:gadget_intersection} are indeed satisfied in our setting: we have $V(G) \cap V(G^*) \neq \emptyset$ and $z^* \notin V(G)$ by assumption, and $z_{\ell} \notin V(G)$ by the definition of a chain and as $m \geq |A_2| \geq 2$.
	
	Suppose now that Condition 1 in the lemma holds. Let $x \in V(G)$ be such that $(z_{\ell},x) \in A(D)$. If $x \in \{z_0,z_1\}$ then the arc $(z_{\ell},x)$ itself constitutes a $(1,b)$-alternating-path $R^*$ with the required properties. Suppose from now on that $x \notin \{z_0,z_1\}$. 
%	So in particular, $(z_0,z_1) \in A_2$ (i.e., $G$ is not a trivial gadget). 
	So in particular, $G$ is not a trivial gadget. 
	Assume first that $G$ is of type I or II. By Item 2 of Lemma \ref{lem:gadget_property_basic}, $\{z_0,z_1\}$ is reachable from $x$ inside $G$. Fix a shortest path $P_0$ from $x$ to $\{z_0,z_1\}$ contained in $G$. 
	Then $|V(P_0) \cap \{z_0,z_1\}| = 1$. Now $R^* := (z_{\ell},x) \circ P_0$ is a dipath from $z_{\ell}$ to $\{z_0,z_1\}$, and hence also a $(1,b)$-alternating-path with $t_1(R^*) \in \{z_0,z_1\}$ and $s_1(R^*) = z_{\ell}$, as required. Assume now that $G$ is of type III.
	Let $r \in V(G)$ and $P_1,P_2 \subseteq V(G)$ be as in the definition of a type-III gadget (so $P_1,P_2$ are dipaths from $z_0,z_1$, respectively, to $r$, each having length at least $2b-1$). Suppose without loss of generality that $x \in V(P_1)$ (the case that $x \in V(P_2)$ is symmetric). Now define a $(2,b)$-alternating-path $R^*$ by setting $s_1(R^*) = z_{\ell}$, $t_1(R^*) = r$, $s_2(R^*) = t_2(R^*) = z_1$, $Q_1(R^*) = (z_{\ell},x) \circ P_1[x,r]$ and $Q'_1(R^*) = P_2$, noting that $|Q'_1(R^*)| = |P_2| \geq 2b-1 \geq b$ by the definition of a type-III gadget. Note that $z_0 \notin V(R^*)$ because $x \notin \{z_0,z_1\}$ by assumption. Thus, $R^*$ satisfies our requirements. 
	
	We have thus shown that for some $1 \leq a_1 \leq 3$, $G \cup G^*$ contains an $(a_1,b)$-alternating-path $R^*$ satisfying
	$t_{a_1}(R^*) \in \{z_0,z_1\}$,
	$s_1(R^*) \in \{z_{\ell},z^*\}$ and 
	$|V(R^*) \cap \{ z_0,z_1 \}| = |V(R^*) \cap \nolinebreak \{z_{\ell},z^*\}| = 1$. Let $R_1$ be the $(a_1,b)$-alternating-path obtained by combining $R^*$ with the dipaths $P[t_{a_1}(R^*),z_{b+1}]$ and $(P\circ (z_\ell,z^*))[z_{\ell-b},s_1(R^*)]$. Formally, we set $s_1(R_1) = z_{\ell-b}$; $t_{a_1}(R_1) = z_{b+1}$; $Q_1(R_1) = \linebreak (P\circ (z_\ell,z^*))[z_{\ell-b},s_1(R^*)] \circ Q_1(R^*)$; $Q_{a_1}(R_1) = Q_{a_1}(R^*) \circ P[t_{a_1}(R^*),z_{b+1}]$; 
	$s_i(R_1) = s_i(R^*)$ for each $2 \leq i \leq a_1$; and $t_i(R_1) = t_i(R^*)$, $Q_i(R_1) = Q_i(R^*)$ and $Q'_i(R_1) = Q'_i(R^*)$ for each $1 \leq i \leq a_1-1$. Note that $R_1$ is strong, i.e., that $Q_1(R_1)$ and $Q_{a_1}(R_1)$ have length at least $b$ \nolinebreak each. 
	
	Put $a_2 := a + 2 - a_1$. Then $1 \leq a_2 \leq a + 1$ because $a \geq 2$ and $1 \leq a_1 \leq 3$. 
	Now set $\mathcal{C}' := \mathcal{C}[z_{b+1},z_{\ell-b}]$, noting that 
	$|A_2(\mathcal{C}')| \geq |A_2(\mathcal{C})| - (2b+1) \geq (a+1)(b+1) - 1 \geq a_2(b+1) - 1$. By Lemma \ref{lem:chain_anti_dir_path}, applied with parameter $a_2$, the chain $\mathcal{C}'$ contains a strong $(a_2,b)$-alternating-path $R_2$ with $s_1(R_2) = z_{b+1} = t_{a_1}(R_1)$ and $t_{a_2}(R_2) = z_{\ell-b} = s_1(R_1)$. Note that 
	$V(R_1) \cap V(R_2) = \{z_{b+1},z_{\ell-b}\}$ because 
	$V(R_1) \subseteq V(G) \cup V(G^*) \cup \{z_0,\dots,z_{b+1}\} \cup \{z_{\ell-b},\dots,z_{\ell}\}$, $V(R_2) \subseteq V(\mathcal{C}') = V(\mathcal{C}[z_{b+1},z_{\ell-b}])$ and 
	$V(\mathcal{C}) \cap V(G^*) \subseteq V(G) \cup \{z_{\ell}\}$, and by the definition of a chain (see Items 3-4 in Definition \ref{def:chain}). By Observation \ref{obs:oriented_cycle}, $R_1 \cup R_2$ spans a subdivision of $C_{a,b}$, as required. 
\end{proof}
	
	\subsection{Embedding Gadgets}
	In this section we prove two lemmas, each asserting that one can find certain gadgets in digraphs $D$ possessing some suitable properties. Recall that $g$ is chosen as $g = 4b^2$. 
	\begin{lemma}\label{lem:I_II_gadget_embedding}
		Let $D$ be a digraph of directed girth at least $g$, and assume that for every $(x,y) \in A(D)$, either $D$ contains a directed cycle of length exactly $g$ through $(x,y)$, or there is $z \in V(D) \setminus \{x,y\}$ such that $(z,x),(z,y) \in A(D)$. Then for every $(p,q) \in A(D)$, there is a type I or extended type-II gadget $G$ contained in $D$ such that $p(G) = p$, $q(G) = q$ and 
		$|V(G)| \leq 2g$.  
	\end{lemma}
	\begin{proof}
%		Recall that $g$ is chosen as $g = 4b^2$. 
		Let $(p,q) \in A(D)$. We inductively define a sequence of vertices $r_i$, $i \geq 0$, with the property that $(r_i,q) \in A(D)$ for every $i \geq 0$ and $(r_i,r_{i-1}) \in A(D)$ for every $i \geq 1$. 
		Set $r_0 := p$. Let $i \geq 1$, and suppose we have already defined $r_0,\dots,r_{i-1}$. By assumption, either $D$ contains a directed cycle $C$ of length exactly $g$ through $(r_{i-1},q)$, or there is $z \in V(D) \setminus \{r_{i-1},q\}$ such that $(z,r_{i-1}),(z,q) \in A(D)$. In the latter case, we set $r_i := z$. In the former case, we stop, noting that $r_{i-1},\dots,r_1,r_0=p,q,C[q,r_{i-1}]$ is a closed directed walk of length $i + (|C| - 1) = g + i - 1$ containing the arc $(p,q)$. It follows that if we stop at step $i$, then there is a directed cycle of length at most $g+i-1$ containing $(p,q)$. Therefore, if the process stopped at step $i$ for some $0 \leq i \leq 2b^2 + b - 2$, then $D$ must contain a directed cycle of length at most $g + 2b^2 + b - 3 \leq 2g$ through $(p,q)$. Moreover, this cycle must have length at least $g$ since the directed girth of $D$ is at least $g$. So we see that in this case, $D$ contains a gadget $G$ of type I with $p(G) = p$, $q(G) = q$ and $|V(G)| \leq 2g$, as required.  
		
		Suppose then that the process carried through to step $2b^2 + b - 2$ (inclusive), and let \linebreak
		$r_0 = p, r_1,\dots,r_{2b^2 + b - 2}$ be the vertices produced by the process. Recall that $r_{2b^2+b-2},\dots,r_1,p$ is a directed walk in $D$, all of whose vertices have an arc to $q$. 
%		Note that $r_0,\dots,r_{2b^2}$ are pairwise-distinct. Indeed, if $r_i = r_j$ for some $0 \leq i < j \leq 2b^2$, then $r_i,r_{i+1},\dots,r_j = r_i$ is a closed walk of length at most $2b^2 < g$, in contradiction to the assumption that 
%		$\overrightarrow{\text{girth}}(D) \geq g$
%%		the directed girth of $D$ is at least $g$ 
%		(here we used our choice of $g$).  
%		Setting $P_1 := r_{2b^2},r_{2b^2-1},\dots,r_1,r_0 = p$, we note that $P_1$ is a dipath of length at least $2b^2$ and that every vertex of $P_1$ has an arc to $q$. Thus, $p,q,P_1$ form a basic gadget of type II. We now show how to extend this gadget. 
		Let $r := r_{2b^2+b-2}$ and $u := r_{2b^2+b-3}$ denote the first and second vertex of this dipath, respectively. We now inductively define a sequence of vertices $w_i$, $i \geq 0$, with the property that $(w_i,u) \in A(D)$ for every $i \geq 0$ and $(w_i,w_{i-1}) \in A(D)$ for every $i \geq 1$. 
		Set $w_0 := r$. Let $i \geq 1$, and suppose we have already defined $w_0,\dots,w_{i-1}$. By assumption, either $D$ contains a directed cycle $C$ of length exactly $g$ through $(w_{i-1},u)$, or there is $z \in V(D) \setminus \{w_{i-1},u\}$ such that $(z,w_{i-1}),(z,u) \in A(D)$. In the latter case, we set $w_i := z$. In the former case, we stop and output the directed cycle $C$. 
		
		Suppose first that the process carried through to step $b$, and let $w_0 = r, w_1,\dots,w_b$ be the vertices produced by the process. Note that 
		$$
		P := (w_b,w_{b-1},\dots,w_0 = r = r_{2b^2+b-2},r_{2b^2+b-3},\dots,r_1,r_0=p,q)
		$$ 
		is a directed walk of length at most $2b^2+2b-1$ in $D$.
%		every two consecutive vertices of which are distinct. 
		Since $\dgirth(D) \geq g = 4b^2 > 2b^2+2b-1$, the vertices of $P$ must be pairwise distinct. Now set $P_1 := (r = r_{2b^2+b-2},\dots,r_1,r_0 = p)$ and $P_2 := (w_b,\dots,w_1,w_0 = r)$,  and observe that $P_1$ and $P_2$ satisfy all the requirements in the definition of an extended type-II gadget (note that there is an arc from the first vertex of $P_2$, namely $w_b$, to the second vertex of $P_1$, namely $u$, by our choice of the vertices $w_i$, $i \geq 0$). Moreover, the resulting gadget has $2b^2+2b \leq 2g$ vertices, as required.  
		
		Suppose now that the process stopped at step $i$ for some $1 \leq i \leq b$, and let $C$ be the outputted directed cycle of length (exactly) $g$ through the arc $(w_{i-1},u)$. As before, the 
		$2b^2 + b + i - 1$ vertices $w_{i-1},\dots,w_0 = r = r_{2b^2+b-2},\dots,r_1,r_0 = p,q$ are pairwise distinct because $\dgirth(D) \geq g = 4b^2 > 2b^2 + b + i - 1$ (as $i \leq b$). 
		Traverse the directed cycle $C$ {\em backwards}, starting from $w_{i-1}$, until the first time a vertex 
		$v \in V := \{w_{i-2},\dots,w_0=r_{2b^2+b-2},\dots,r_1,p,q\}$ is hit (this will surely happen because $u = r_{2b^2+b-3} \in V(C) \cap V$). By our choice of $v$ we have $V(C[v,w_{i-1}]) \cap V = \{v\}$. We now rule out the possibility that $v \in \{w_0,\dots,w_{i-2}\}$. To this end, suppose by contradiction that $v = w_j$ for some $0 \leq j \leq i-2$. Since $w_{i-1},w_{i-2},\dots,w_j = v,C[v,w_{i-1}],w_{i-1}$ is a directed cycle and $\dgirth(D) \geq g$, it must be the case that
		$|C[v,w_{i-1}]| \geq g - (i - 1 - j) \geq g - b + 1$. Now, as $C$ consists of the arc $(w_{i-1},u)$ and the dipaths $C[v,w_{i-1}]$ and $C[u,v]$, we have 
		$|C[u,v]| = |C| - 1 - |C[v,w_{i-1}]| = g - 1 - |C[v,w_{i-1}]| \leq b - 2$. Finally, we get that 
%		$v=w_j,w_{j-1},\dots,w_0 = r,u,C[u,v],v$ 
		$v=w_j,(w_j,u),u,C[u,v],v = w_j$ 
		is a (non-trivial) directed closed walk of length at most $b-1 < g$, a contradiction. 
		
		We have thus shown that $v \notin \{w_0,\dots,w_{i-2}\}$. If $v = q$ then $w_{i-1},u = r_{2b^2+b-3},\dots,r_1,r_0=p,q = v,C[v,w_{i-1}]$ is a directed cycle which goes through the arc $(p,q)$ and has length at most $g + 2b^2 + b - 1 \leq 2g$ and at least $g$. Hence, in this case $D$ contains a gadget of type I with the required properties.
		It remains to handle the case that $v \in \{u = r_{2b^2+b-3},\dots,r_1,r_0 = p\}$. In this case, let $s$ be the second vertex of $C[v,w_{i-1}]$ (so in particular, $(v,s) \in A(D)$). Now define the dipaths $P_1 := (r_{2b^2+b-2} = r,\dots,r_1,r_0 = p)$
		and $P_2 := C[s,w_{i-1}] \circ (w_{i-1},\dots,w_1,w_0 = r)$. 
		We claim that $|P_2| \geq b$. Indeed, since $(v,s) \circ P_2 \circ P_1[r,v]$ is a directed cycle in $D$ and $\dgirth(D) \geq g$, it must be the case that $|P_2| \geq g - 1 - |P_1| = g - 1 - (2b^2 + b - 2) \geq b$. 
%		We claim that $P_1$ and $P_2$ satisfy all requirements in the definition of an extended type-II gadget. The only requirement which has not yet been established is that $|P_2| \geq b$. 
		We now see that all of the requirements in the definition of an extended type-II gadget are met (note that the vertex $v \in V(P_1) \setminus \{r\}$ has an arc to the first vertex of $P_2$, namely $s$). Finally, observe that the resulting type-II gadget has 
		$2b^2 + b + (i-1) + |C[s,w_{i-1}]| \leq 2b^2 + 2b + g \leq 2g$ vertices, as required.    
	\end{proof}

	\begin{lemma}\label{lem:III_gadget_embedding}
	Let $b,h,d \geq 1$ be integers, let $D'$ be a digraph and let $v \in V(D')$. Suppose that the following two conditions hold.
	\begin{enumerate}
%		\item Every vertex of $D$ reachable from $v$ has out-degree at least $\delta$.
		\item Every vertex of $D'$ reachable from $v$ has out-degree at least 
		$(h+1) \cdot (d(2b-2) + 1) + d$;
		\item The number of vertices of $D'$ at distance at most $(h+1)(2b-1)$ from $v$ is less than 
%		$\frac{d^{h+1}-1}{d-1} \cdot \big(d(2b-2) + 1\big)$.
		$d^h$.  
%		$C \cdot \Delta^{\ell}$. 
	\end{enumerate}
	Then $D'$ contains a type-III gadget $G$ and a dipath $P_0$ from $v$ to $p(G)$ such that $V(P_0) \cap V(G) = \{p(G)\}$, $|V(G)| \leq (2h+2)(2b-1)$ and $|V(P_0)| \leq h(2b-1)$. 
%	\newline
%	??? Do we actually need that $|P| \leq h-2$ (for using Lemma \ref{lem:III_gadget_embedding} in the proof of Theorem \ref{thm:cycle_orientation_main})? ???
\end{lemma}
\begin{proof}
		We describe a process for producing a (specific) out-arborescence $T \subseteq D$ with root $v$. 
		The idea is as follows: going level by level (in a breadth-first manner), we will try to attach to each vertex $u$ of the (current) lowest level a collection of $d$ dipaths of length $2b-1$ each, which intersect only at $u$ and do not intersect the (current) tree in any other vertex. In this manner, we will construct a {\em $(2b-1)$-subdivision} of a {\em $d$-ary out-arborescence}, where an $s$-subdivision of a digraph $F$ is a subdivision of $F$ in which every arc is replaced with a dipath of length (exactly) $s$, and a $d$-ary out-arborescence is an out-arborescence in which every non-leaf vertex has exactly $d$ children. We will then use Item 2 to argue that rather soon in this process, intersections of branches must occur. Such an intersection will give rise to the desired type-III gadget. The \nolinebreak details \nolinebreak follow. 
		
		Throughout the process, we will maintain and update an out-arborescence $T$ and sets $L_i$, $i \geq 0$. 
		We start by setting $L_0 = \{v\}$ and initializing $T$ to be the one-vertex tree with root $v$. 
%		(the tree $T$ will be updated along the process). 
		Let $i \geq 0$, and suppose that we have already defined $L_0,\dots,L_i$. If $L_i = \emptyset$ then we stop and say that the process terminated at step $i$.
		Otherwise, initialize $L_{i+1}$ to be the empty set and proceed as follows. Let $u_1,\dots,u_t$ be an enumeration of the vertices in $L_i$. Going over $j = 1,\dots,t$ in increasing order, we let $\mathcal{P}(u_j)$ be the set of all dipaths of length $2b-1$ which start at $u_j$ and are otherwise disjoint from $V(T)$. If $\mathcal{P}(u_j)$ contains $d$ dipaths $Q_1,\dots,Q_d$ with 
		$V(Q_k) \cap V(Q_{\ell}) = \{u_j\}$ for all $1 \leq k < \ell \leq d$, then attach these dipaths to $T$ and add their endpoints to $L_{i+1}$. Otherwise, i.e. if $\mathcal{P}_j$ does not contain $d$ dipaths $Q_1,\dots,Q_d$ which pairwise intersect only at $u_j$, then do nothing; in this case $u_j$ will remain a leaf of $T$ throughout the \nolinebreak process.
			
		Consider the out-arborescence $T$ at the end of the process. 	
		It is easy to see that $T$ is indeed the $(2b-1)$-subdivision of some $d$-ary out-arborescence $T_0$, and that the branch vertices of this subdivision are precisely the elements of $\bigcup_{i \geq 0}{L_i}$. It follows that $|L_i| \leq d^i$ for every $i \geq 0$.  
		
		We claim that there is $0 \leq i \leq h$ such that $L_i$ contains a leaf of $T$. Indeed, suppose by contradiction that for every $0 \leq i \leq h$, no vertex of $L_i$ is a leaf. Then $|L_i| = d^i$ for every $0 \leq i \leq h+1$. Observe that 
		the number of vertices of $T$ which are at distance at most $(h+1)(2b-1)$ from $v$ (in $T$) is exactly $|L_{h+1}| + \sum_{i=0}^{h}{\big(d(2b-2) + 1\big) \cdot |L_i|}$. Hence, the number of such vertices is at least
		$$
		\sum_{i=0}^{h}{\big(d(2b-2) + 1\big) \cdot |L_i|} = \big(d(2b-2) + 1\big) \cdot \sum_{i=0}^{h}{d^i} 
%		= 
%		\frac{d^{h+1}-1}{d-1} \cdot \big(d(2b-2) + 1\big) 
		\geq 
		d^h,
		$$
		in contradiction to the assumption in Item 2 of the lemma.
		
		So let $0 \leq i \leq h$ be such that $L_i$ contains a leaf of $T$, and let $u \in L_i$ be such a leaf. 
		Let $X$ be the set of vertices of $T$ which are at distance at most $(i+1) \cdot (2b-1)$ from the root $v$. In other words, $X$ consists of the sets $L_0,\dots,L_{i+1}$ and the (vertices of the) subdivision dipaths connecting $L_j$ to $L_{j+1}$ for $j = 0,\dots,i$. 
		Say that a $d$-tuple of dipaths $(Q_1,\dots,Q_d)$ is {\em good} if
		\begin{enumerate}
			\item[(a)] For every $k=1,\dots,d$, it holds that $|Q_k| \leq 2b-1$,  $V(Q_k) \cap X = \{u\}$, and $u$ is the first vertex of $Q_k$.
			\item[(b)] $V(Q_k) \cap V(Q_{\ell}) = \{u\}$ for all $1 \leq k < \ell \leq d$.
		\end{enumerate}
		Among all good $d$-tuples of dipaths, let $(Q_1,\dots,Q_d)$ be one which maximizes $|Q_1| + \dots + |Q_d|$ (note that taking $Q_1,\dots,Q_d$ to be empty dipaths (starting at $u$) gives a good $d$-tuple, so the set of good $d$-tuples is non-empty). Observe that if we had $|Q_1| = \dots = |Q_d| = 2b-1$, then the algorithm would have attached $Q_1,\dots,Q_d$ to $T$, in contradiction to our assumption that $u$ is a leaf. Thus, there must be some $1 \leq k \leq d$ such that $|Q_k| \leq 2b-2$. Suppose without loss of generality  that $|Q_1| \leq 2b-2$, and let $w$ be the last vertex of $Q_1$. Evidently, $w$ is reachable from $u$ and hence also from $v$, implying that 
		$d^+(w) \geq (h+1) \cdot (d(2b-2) + 1) + d$ by Item 1. If $w$ had an out-neighbour in 
		$V(D') \setminus (X \cup V(Q_1) \cup \dots \cup V(Q_d))$, then we could extend $Q_1$ and thus obtain a longer good $d$-tuple of dipaths, in contradiction to the maximality of $(Q_1,\dots,Q_d)$. Thus, $N^+(w) \subseteq X \cup V(Q_1) \cup \dots \cup V(Q_d)$. 
		As $|N^+(w) \cap (V(Q_1) \cup \dots \cup V(Q_d))| \leq |V(Q_1) \cup \dots \cup V(Q_d)| - 1 = |V(Q_1)| + \dots + |V(Q_d)| - (d-1) - 1 \leq d \cdot 2b - 1 - d 
		= d(2b-1) - 1$, we must have 
%		$|N^+(w) \cap X| \geq hd(2b-2) + h + 1$. 
		$|N^+(w) \cap X| \geq d^+(w) - d(2b-1) + 1 \geq hd(2b-2) + h + 2$. 
		
		For each vertex $x \in X$, let $y(x)$ denote the lowest common ancestor of $u$ and $x$ in the tree \nolinebreak $T$. 
		Let $X'$ be the set of all vertices $x \in X$ such that (at least) one of the vertices $u,x$ is at distance at most $2b-2$ from $y(x)$ in $T$. We now show that $|X'| \leq h d (2b-2) + h + 1$. 
		Let $P$ be the unique dipath (in $T$) from $v$ to $u$. For each $0 \leq j \leq i-1$, let $y_j$ be the unique element of $V(P) \cap L_j$. Observe that if $x \in X'$, then either $x \in V(P)$, or there is $0 \leq j \leq i-1$ such that $x$ is an internal vertex of one of the $d-1$ subdivision dipaths which start at $y_j$ and are not subpaths of $P$. (Recall that every non-leaf branching vertex of $T$ is the first vertex of exactly $d$ subdivision dipaths. It is evident that for every $0 \leq j \leq i-1$, exactly one of the $d$ subdivision dipaths starting at $y_j$ is a subpath of $P$, while the other $d-1$ only intersect $P$ at $y_j$.) It follows that 
		$|X'| = |V(P)| + i \cdot (d-1) \cdot (2b-2) = i \cdot (2b-1) + 1 + i \cdot (d-1) \cdot (2b-2) = i \cdot d \cdot (2b-2) + i + 1 \leq h d (2b-2) + h + 1$, 
		as claimed. 
		
%		Let $X^*$ be the set of all $x \in X \setminus \{u\}$ for which there is $P \in \mathcal{P}$ such that $x \in V(P)$.
		As $|N^+(w) \cap X| \geq h d (2b-2) + h + 2 > |X'|$, there exists $x \in N^+(w) \cap (X \setminus X')$. 
		Setting $y = y(x)$, let $P'_1$ (resp. $P'_2$) be the unique dipath (in $T$) from $y$ to $x$ (resp. $u$). Since $x \notin X'$, we have $|P'_1|,|P'_2| \geq 2b-1$. 
		As $u \in L_i$, we have $|P'_2| \leq i(2b-1) \leq h(2b-1)$, and as $x \in X$ we have $|P'_1| \leq (h+1)(2b-1)$. 
		Let $z$ be the second vertex of $P'_2$.  
%		Now set $P_0 := P[v,y]$, $P_1 :=  P'_1[z,u] \circ Q_1 \circ (w,x)$ and $P_2 := P'_2$. 
		Now set $P_0 := P[v,y]$, $P_1 := P'_1$ and $P_2 :=  P'_2[z,u] \circ Q_1 \circ (w,x)$. 
		Observe that $P_0,P_1,P_2$ are internally vertex-disjoint (as $y$ is the lowest common ancestor of $x$ and $u$), and that $P_1$ and $P_2$  have length at least $2b-1$ each 
		(indeed, we have $|P_1| = |P'_1| \geq 2b-1$ and $|P_2| \geq |P'_2| - 1 + |Q_1| + 1 \geq |P'_2| \geq 2b-1$). 
		So we see that $P_1,P_2$ form a type-III gadget $G$ with $p(G) = y$ and $q(G) = z$, and that this gadget satisfies $V(P_0) \cap V(G) = \{y\} = \{p(G)\}$. Finally, observe that
		$|P_0| \leq (i-1)(2b-1) \leq h(2b-1)$, $|P_1| = |P'_1| \leq (h+1)(2b-1)$ and $|P_2| = (|P'_2| - 1) + |Q_1| + 1 \leq h(2b-1) + 2b-2 \leq (h+1)(2b-1)$. It follows that $|V(G)| = |V(P_1)| + |V(P_2)| - 1 \leq (2h+2)(2b-1)$. This completes the proof. 
\end{proof}

\subsection{Putting It All Together}
	\begin{proof}[Proof of Theorem \ref{thm:cycle_orientation_main}]
%		As mentioned before, the case $a = 1$ of Theorem \ref{thm:cycle_orientation_main} was established in \cite{aboulker}. Thus, we will only prove the theorem for $a \geq 2$. 
		Let $a \geq 2$ and $b \geq 1$. Recall that we set 
		$g := 4b^2$. We also fix an integer $k = O(ab^7)$, to be chosen later. 
		Suppose, for the sake of contradiction, that the theorem is false, and let $D$ be a counterexample to the theorem which minimizes $|V(D)| + |A(D)|$. Namely, we assume that $\delta^+(D) \geq k$, $\dgirth(D) \geq g$ and $D$ does not contain a subdivision of $C_{a,b}$, but every digraph $D'$ with $|V(D')| + |A(D')| < |V(D)| + |A(D)|$, $\delta^+(D') \geq k$ and $\dgirth(D') \geq g$ does contain a subdivision of $C_{a,b}$. 
		
		\paragraph{Claim 1.} $d^+(v) = k$ for every $v \in V(D)$.
		\begin{proof}
			Suppose, by contradiction, that $d^+(v) \geq k+1$ for some $v \in V(D)$. Let $D'$ be the digraph obtained from $D$ by deleting an (arbitrary) arc whose tail is $v$. Then $\delta^+(D') \geq k$ and $\dgirth(D') \geq g$, but $D'$ does not contain a subdivision of $C_{a,b}$ (as $D'$ is a subgraph of $D$). This contradicts the minimality of $D$.  
		\end{proof}
		
		\paragraph{Claim 2.} For every $(x,y) \in A(D)$, either $D$ contains a directed cycle of length exactly $g$ through $(x,y)$, or there is $z \in V(D) \setminus \{x,y\}$ such that $(z,x),(z,y) \in A(D)$.
		\begin{proof}
			Let $(x,y) \in A(D)$. 
			Suppose by contradiction that the assertion of the claim is false. 
			Let $D'$ be the digraph obtained from $D$ by deleting $x$ and adding the arc $(z,y)$ for every $z \in N_D^-(x)$. 
			Evidently, $|V(D')| + |A(D')| < |V(D)| + |A(D)|$. 
			We claim that $\delta^+(D') \geq k$ and 
			$\dgirth(D') \geq g$. 
			First, note that $d^+_{D'}(y) = d^+_{D}(y) = k$ because $(y,x) \notin A(D)$ (as $\dgirth(D) \geq g > 2$). Next, observe that for every $z \in V(D') \setminus \{y\} = V(D) \setminus \{x,y\}$ we also have $d^+_{D'}(z) = d^+_{D}(z) = k$, because $z$ does not have both $x$ and $y$ as out-neighbors (by our assumption). It follows that $\delta^+(D') \geq k$. 
			Now suppose, for the sake of contradiction, that $D'$ contains a directed cycle $C'$ of length at most $g-1$. If there is no $z \in N_D^-(x)$ such that $(z,y) \in A(C')$, then $C'$ is also contained in $D$, which is impossible as $\dgirth(D) \geq g$. So let $z \in N_D^-(x)$ be such that $(z,y) \in A(C')$, and let $C$ be the directed cycle obtained from $C'$ by deleting the arc $(z,y)$ and adding the arcs $(z,x),(x,y)$. Then $C$ is contained in $D$ and has length $|C'| + 1 \leq g$, implying that $|C| = g$. But this is impossible as we assumed that $D$ contains no directed cycle of length $g$ through the arc $(x,y)$. We conclude that $\dgirth(D') \geq g$, as claimed. 			
			
			The minimality of $D$ implies that $D'$ contains a subdivision $S'$ of $C_{a,b}$. If there is no $z \in N_D^-(x)$ such that $(z,y) \in A(S')$, then $S'$ is also contained in $D$, contradicting our assumption that $D$ contains no subdivision of $C_{a,b}$. 
			Suppose then that the set $Z := \{z \in N_D^-(x) : (z,y) \in A(S')\}$ is non-empty. Since the maximum in-degree of $C_{a,b}$ is $2$, we have $|Z| \leq 2$. 
			Assume first that $|Z| = 1$, and write $Z = \{z\}$. By replacing the edge $(z,y)$ of $S'$ with the path $(z,x),(x,y)$ (which is present in $D$), we obtain a subdivision of $C_{a,b}$ contained in $D$, a contradiction. Suppose now that $|Z| = 2$, and write $Z = \{z_1,z_2\}$. Then $y$ must be a branch vertex in $S'$, and we must have $d^+_{S'}(y) = 0$ (since every branch vertex of $C_{a,b}$ is either a source or a sink). Let $S$ be the subgraph of $D$ obtained from $S'$ by deleting the edges $(z_1,y),(z_2,y)$ and adding the edges $(z_1,x),(z_2,x)$. Then $S$ is a subdivision of $C_{a,b}$ in which $x$ plays the branch-vertex role played in $S'$ by $y$. Again, we have arrived at a contradiction to our assumption that $D$ contains no subdivision of $C_{a,b}$.
		\end{proof}
		
%		Fix an integer $h$, to be chosen later, such that $g \leq (h+1)(2b-1)$. 
%		Fix an integer $h \geq 2b+1$, to be chosen later.
		Let $\mathcal{C}$ be a chain with spine $P = v_0,\dots,v_m$ and partition $A(P) = A_1 \cup A_2$ (as in Definition \nolinebreak \ref{def:chain}). We say that $\mathcal{C}$ is {\em good} if the following conditions are satisfied:
		\begin{enumerate}
			\item[(a)] Every gadget in $\mathcal{C}$ has at most $(8g+6)(2b-1)$ vertices;
			\item[(b)] $(v_{m-1},v_m) \in A_2$;
			\item[(c)] Among any $(4g+3)(2b-1)$ consecutive arcs of $P$, there is an arc belonging to $A_2$.  
		\end{enumerate}
		Among all good chains contained in $D$, let $\mathcal{C}$ be one of maximal length, and let $P = v_0,\dots,v_m$, $A_1,A_2$ and $(G_e)_{e \in A_2}$ be as in Definition \ref{def:chain}. 
		Define 
		$i_0 := \max\big\{0, m - (4g+3)(2b-1)(a+3)(b+1)\big\}$ and
		$
		\mathcal{C}' := \mathcal{C}[v_{i_0},v_m].
		$
		Item (a) implies that
%		and note that 
%		$V(\mathcal{C}') \leq h(a+2)b + 2 + (2g-2) \cdot (h(a+2)b + 1) = (2g-1)h(a+2)b + 2g \leq 8ghab$ 
%		$|V(\mathcal{C}')| \leq 2g \cdot (h-1)(a+2)(b+1) \leq 8ghab$
%		(in the first inequality we used Item (a)).  
%		\begin{equation}\label{eq:size_of_C'}
%		|V(\mathcal{C}')| \leq \max\{2g,(2h+2)(2b-1)\} \cdot (m-i) = 
%		O\left( (g+hb)ahb^2 \right) = O\left( (b^2+hb)ahb^2 \right).
%		\end{equation} 
		\begin{equation}\label{eq:size_of_C'}
		\begin{split}
		|V(\mathcal{C}')| \leq (8g+6)(2b-1) \cdot (m-i_0) &\leq 
		2(4g+3)^2(2b-1)^2(a+3)(b+1) 
		\\ &\leq 
		8b^2(4g+3)^2(a+3)(b+1).
		\end{split} 
		\end{equation} 
%		Here, the inequality follows from Item (a).  
		Our choice of $i_0$ implies that either $i_0=0$ and $\mathcal{C}'=\mathcal{C}$, or $i_0 = m - (4g+3)(2b-1)(a+3)(b+1)$, in which case we have by Item (c) that
		\begin{equation*}\label{eq:richness_of_C'}
		|A_2(\mathcal{C}')| \geq \left\lfloor \frac{m-i_0}{(4g+3)(2b-1)} \right\rfloor = (a+3)(b+1).
		\end{equation*}
		Let $D'$ be the digraph obtained from $D$ by deleting the vertex-set $V(\mathcal{C}) \setminus \{v_m\}$.
		
		\paragraph{Claim 3.} Every $u \in V(D')$ which is reachable from $v_m$ in $D'$ satisfies $d^+_{D'}(u) \geq k - |V(\mathcal{C}')|$.
		\begin{proof}
		If $i_0=0$, then $\mathcal{C}'=\mathcal{C}$ and the claim follows directly by definition of $D'$. So assume now that $i_0>0$ and hence $|A_2(\mathcal{C'})| \ge (a+3)(b+1) \geq (a+3)(b+1)-2$.
			Let $Q$ be a dipath from $v_m$ to $u$ in $D'$. Suppose by contradiction that $d^+_{D'}(u) < k - |V(\mathcal{C}')|$. Since $d^+_{D}(u) \geq k$ and $V(D) \setminus V(D') \subseteq V(\mathcal{C})$, we must have 
			$|N^+_D(u) \cap V(\mathcal{C})| > |V(\mathcal{C}')|$.  
			Hence, there must be some $x \in V(\mathcal{C}) \setminus V(\mathcal{C}')$ such that $(u,x) \in A(D)$. Let $0 \leq j \leq m$ be such that $x \in G_{(v_j,v_{j+1})}$, and note that $j < i_0$ because $x \notin V(\mathcal{C}')$. Now let $\mathcal{C}''$ be the chain formed by concatenating $\mathcal{C}[v_j,v_m]$ with the dipath $Q$ (in this chain, all arcs of $Q$ belong to $A_1(\mathcal{C}'')$). 
			This is indeed a chain because $V(Q) \cap V(\mathcal{C}) = \{v_m\}$.
			As $\mathcal{C}'$ is contained in $\mathcal{C}''$, we have $|A_2(\mathcal{C}'')| \geq |A_2(\mathcal{C}')| \geq (a+3)(b+1) - 2$. Observe that we are precisely in the setting of Item 1 of Lemma \ref{lem:chain_main} with respect to the chain $\mathcal{C}''$. Indeed, the last vertex of the spine of $\mathcal{C}''$, namely $u$, sends an arc to $x \in V(G_{(v_j,v_{j+1})})$, and $(v_j,v_{j+1})$ is the first arc of the spine of $\mathcal{C}''$. So we may apply Item 1 of Lemma \ref{lem:chain_main} to deduce that $D$ contains a subdivision of $C_{a,b}$, a contradiction. 
		\end{proof}
		
		\paragraph{Claim 4.}
		Let $u \in V(D')$ be a vertex whose distance from $v_m$ in $D'$ is at most $(4g+3)(2b-1)$. Then $D$ contains a dipath of length at most $2g$ from $V(\mathcal{C}')$ to $u$.   
		\begin{proof}
			Let $Q = (w_0 = v_m,w_1,\dots,w_{t-1},w_t = u)$ be a shortest dipath from $v_m$ to $u$ in $D'$. Then $t = |Q| \leq (4g+3)(2b-1)$. 
%			Suppose, for the sake of contradiction, that $D$ contains no dipath of length at most $2g$ from $V(\mathcal{C}')$ to $u$. 
			Claim 2 states that $D$ satisfies the condition of Lemma \ref{lem:I_II_gadget_embedding}. 
			By applying Lemma \ref{lem:I_II_gadget_embedding} to the arc $(w_{t-1},u)$, we infer that $D$ contains a gadget $G^*$ which is either of type I or extended type-II, such that $p(G^*) = w_{t-1}$, $q(G^*) = u$ and $|V(G^*)| \leq 2g$. 
			
			We now show that if $V(G^*) \cap V(\mathcal{C}') \neq \emptyset$, then the assertion of the claim holds. So suppose that $V(G^*) \cap V(\mathcal{C}') \neq \emptyset$, and let $x \in V(G^*) \cap V(\mathcal{C}')$. By Item 2 of Lemma \ref{lem:gadget_property_basic}, $G^*$ contains a dipath from $x$ to $\{w_{t-1},u\}$, and hence also to $u$, as $(w_{t-1},u) \in A(G^*)$. Evidently, this dipath has length at most $|V(G^*)| \leq 2g$. So we see that $D$ contains a dipath of length at most $2g$ from $V(\mathcal{C}')$ to $u$, as required. 
			To complete the proof, it hence suffices to show that $V(G^*) \cap V(\mathcal{C}') \neq \emptyset$. For the rest of the proof we assume, for the sake of contradiction, that $V(G^*) \cap V(\mathcal{C}') = \emptyset$.
			We proceed by a case analysis over the type of $G^*$.
			
			\paragraph{Case 1.} $G^*$ is an extended gadget of type II. Let $G^*_0$ be the basic part of $G^*$. 
			We claim that $V(G^*_0) \cap V(Q) = \{w_{t-1},u\}$. Suppose otherwise, and let $0 \leq j \leq t-2$ be such that $w_j \in V(G^*_0)$. By the definition of a basic type-II gadget, every vertex in $V(G^*_0)\setminus \{u\}$ has an arc to $q(G^*_0) = u$. In particular, $(w_j,u) \in A(D)$, and hence also $(w_j,u) \in A(D')$ (as $w_j,u \in V(D')$). It follows that $w_0,\dots,w_{j-1},w_j,u$ is a dipath from $w_0 = v_m$ to $u$ in $D'$ which is shorter than $Q$, in contradiction to our choice of $Q$. So indeed we have $V(G^*_0) \cap V(Q) = \{w_{t-1},u\}$. 
			
			We claim that $V(G^*) \cap V(\mathcal{C}) \neq \emptyset$. So suppose by contradiction that 
			$V(G^*) \cap V(\mathcal{C}) = \emptyset$.
			% or $V(G^*) \cap V(\mathcal{C}) = \{v_m\}$
			% or $V(G^*) \cap V(\mathcal{C}) \subseteq \{v_m\}
			Then one can extend the chain $\mathcal{C}$ into a longer good chain $\mathcal{C}_1$ by adding the dipath $Q$ and the gadget $G^*_0$; the definition of $\mathcal{C}_1$ includes setting $(w_{t-1},u) \in \nolinebreak A_2(\mathcal{C}_1)$, $G_{(w_{t-1},u)}(\mathcal{C}_1) = G^*_0$,
			and $(w_j,w_{j+1}) \in A_1(\mathcal{C}_1)$ for every $0 \leq j \leq t-2$. Then 
			$\mathcal{C}_1$ is indeed a chain because $V(G^*_0) \cap V(Q) = \{w_{t-1},u\}$ and due to our assumption that $V(G^*) \cap V(\mathcal{C}) = \emptyset$.
			The goodness of $\mathcal{C}_1$ (i.e. that $\mathcal{C}_1$ satisfies Items (a)-(c) above) follows from the goodness of $\mathcal{C}$ and the fact that $|Q| \leq (4g+3)(2b-1)$ and $|V(G^*_0)| \leq 2g \le (8g+6)(2b-1)$. As the existence of $\mathcal{C}_1$ stands in contradiction to the maximality of \nolinebreak $\mathcal{C}$, our assumption $V(G^*) \cap V(\mathcal{C}) = \emptyset$ must have been wrong, as required. 
			
			We have thus shown that $V(G^*) \cap V(\mathcal{C}) \neq \emptyset$. Since $V(G^*) \cap V(\mathcal{C}') = \emptyset$ by assumption, we must have $V(G^*) \cap (V(\mathcal{C}) \setminus V(\mathcal{C}')) \neq \emptyset$. This means that $V(G^*) \cap V(G_{(v_i,v_{i+1})}) \neq \emptyset$ for some $0 \leq i < i_0$ (as $V(\mathcal{C}) \setminus V(\mathcal{C}')$ is contained in the union of $V(G_{(v_i,v_{i+1})})$ over all $0 \leq i < i_0$). Let $i_1$ be the largest such $0 \leq i < i_0$, and set $G := G_{(v_{i_1},v_{i_1+1})}$.  
			Now let $\mathcal{C}_1$ be the chain obtained by attaching to $\mathcal{C}[v_{i_1},v_m]$ the dipath 
			$Q - u = (w_0 = v_m,w_1,\dots,w_{t-1})$. This is indeed a chain because $V(Q) \cap V(\mathcal{C}) = \{v_m\}$ (as $V(Q) \subseteq V(D')$ and $(V(\mathcal{C}) \setminus \{v_m\}) \cap V(D') = \emptyset$). Then $|A_2(\mathcal{C}_1)| \geq |A_2(\mathcal{C}')| \geq (a+3)(b+1) - 2$ because $\mathcal{C}_1$ contains $\mathcal{C}'$ and $i_0>0$. Observe that Condition 2 in Lemma \ref{lem:chain_main} holds for the chain $\mathcal{C}_1$ with respect to the vertex $z^* := u$ (and with $z_{\ell} = w_{t-1}$, $z_0 = v_{i_1}$ and $z_1 = v_{i_1+1}$). Indeed, there is an arc from the last vertex of the spine of $\mathcal{C}_1$, namely $w_{t-1}$, to $u \notin V(\mathcal{C}_1)$, and there is an extended type-II gadget $G^*$ such that $p(G^*) = w_{t-1}$, $q(G^*) = u$,  $V(G) \cap \nolinebreak V(G^*) \neq \emptyset$ and 
			$V(\mathcal{C}_1) \cap V(G^*) \subseteq V(G) \cup \{w_{t-1}\}$ (here we use our choice of $i_1$). By Lemma \ref{lem:chain_main}, $D$ contains a subdivision of $C_{a,b}$, a contradiction. 
			
			\paragraph{Case 2.} $G^*$ is of type I, i.e., a directed cycle of length at least $g$ through $(w_{t-1},u)$. 
%			We proceed similarly to the previous case, but replace the dipath $Q$ with a (possibly) different dipath $Q'$. 
			Let $j$ be the smallest integer in $\{0,\dots,t-1\}$ satisfying $w_j \in V(G^*)$; note that $j$ is well-defined because $w_{t-1} \in V(G^*)$.
			Let $w'$ be the vertex of the directed cycle $G^*$ immediately following $w_j$, and 
%			put $Q' := Q[w_0 = v_m,w_j] \circ (w_j,w') = (w_0,w_1,\dots,w_j,w')$.
			consider the dipath $Q' := (w_0 = v_m,w_1,\dots,w_j,w')$. 
			Our choice of $j$ implies that \linebreak $w' \notin \{w_0,\dots,w_{j-1}\}$ (so $Q'$ is indeed a path) and that $V(G^*) \cap V(Q') = \{w_j,w'\}$. 
			Note also that $j \geq 1$ because $w_0 = v_m \in V(\mathcal{C}')$ and $V(G^*) \cap V(\mathcal{C}') = \emptyset$ by assumption. Hence, $w_j \notin V(\mathcal{C})$.
			
			Similarly to the previous case, if $V(G^*) \cap V(\mathcal{C}) = \emptyset$, 
			% or $V(G^*) \cap V(\mathcal{C}) = \{v_m\}$
			% or $V(G^*) \cap V(\mathcal{C}) \subseteq \{v_m\}
			then one can extend $\mathcal{C}$ into a longer good chain $\mathcal{C}_1$ by adding the dipath $Q'$ and the gadget $G^*$; the definition of $\mathcal{C}_1$ includes setting $(w_j,w') \in A_2(\mathcal{C}_1)$, $G_{(w_j,w')}(\mathcal{C}_1) = G^*$, and $(w_i,w_{i+1}) \in A_1(\mathcal{C}_1)$ for every $0 \leq i \leq j-1$. Then 
			$\mathcal{C}_1$ is indeed a chain because $V(G^*) \cap V(Q') = \{w_j,w'\}$ and $V(G^*) \cap V(\mathcal{C}) = \emptyset$, and the goodness of $\mathcal{C}_1$ follows from the goodness of $\mathcal{C}$ and the fact that $|Q| \leq (4g+3)(2b-1)$ and \linebreak $|V(G^*)| \leq 2g \le (8g+6)(2b-1)$. So we see that having 
			$V(G^*) \cap V(\mathcal{C}) = \emptyset$ contradicts the maximality of $\mathcal{C}$, and hence $V(G^*) \cap V(\mathcal{C}) \neq \emptyset$.
			
%		Assume first that $V(G^*) \cap V(\mathcal{C}) \neq \emptyset$. 
		Walk along the directed cycle $G^*$, starting from $w_j$, until the first time that a vertex of $V(\mathcal{C})$ is met. Denote this vertex by $x$, and the preceding vertex on $G^*$ by $y$. Consider the dipath $Q'' := (w_0 = v_m,w_1,\dots,w_j) \circ G^*[w_j,y]$, and observe that $V(Q'') \cap V(\mathcal{C}) = \{v_m\}$ because $V(Q) \cap V(\mathcal{C}) = \{v_m\}$ and by our choice of $x$. 
		Since $V(G^*) \cap V(\mathcal{C}') = \emptyset$, we must have 
		$x \in V(G^*) \cap (V(\mathcal{C}) \setminus V(\mathcal{C}')) \neq \emptyset$. This means that $x \in V(G_{(v_i,v_{i+1})})$ for some 
		$0 \leq i < i_0$. 
%		(as $V(\mathcal{C}) \setminus V(\mathcal{C}')$ is contained in the union of $V(G_{(v_i,v_{i+1})})$ over all $0 \leq i < i_0$). 
		Now let $\mathcal{C}_1$ be the chain obtained by concatenating $\mathcal{C}[v_{i},v_m]$ with the dipath 
		$Q''$. This is indeed a chain because $V(Q'') \cap V(\mathcal{C}) = \{v_m\}$ . Then $|A_2(\mathcal{C}_1)| \geq |A_2(\mathcal{C}')| \geq (a+3)(b+1) - 2$ because $\mathcal{C}_1$ contains $\mathcal{C}'$ and $i_0>0$. 
		Observe that Condition 1 in Lemma \ref{lem:chain_main} holds for the chain $\mathcal{C}_1$ (with $y$ playing the role of $z_{\ell}$). Indeed, there is an arc from the last vertex of the spine of $\mathcal{C}_1$, namely $y$, to $x \in V(G_{(i,i+1)})$, and $(v_i,v_{i+1})$ is the first arc of the spine of $\mathcal{C}_1$. By Lemma \ref{lem:chain_main}, $D$ contains a subdivision of $C_{a,b}$, a contradiction. This completes the proof of Claim 4. 
		\end{proof}
		
		With Claims 3-4 at hand, we can complete the proof of the theorem. To this end, we will apply Lemma \ref{lem:III_gadget_embedding}. 
		By combining Claim 4 with the fact that $\Delta^+(D) = k$, we conclude that the number of vertices of $D'$ at distance at most $(4g+3)(2b-1)$ from $v_m$ (in $D'$) is at most $|V(\mathcal{C}')| \cdot k^{2g}$. 
		We will apply Lemma \ref{lem:III_gadget_embedding} with parameters $h := 4g+2$ and $d := 2b(4g+3)(a+3)(b+1)$. 
		To this end, we will need to verify that
		\begin{equation}\label{eq:choice_of_h}
		k - |V(\mathcal{C}')| \geq (4g+3) \cdot (d(2b-2) + 1) + d \text{ and } 
		|V(\mathcal{C}')| \cdot k^{2g} < 
		d^{4g+2}. 
		\end{equation}
%		Recalling that $g = 4b^2$, choose 
%%		$h := \lceil 2g\log_2 k + \log_2 |V(\mathcal{C}')| - \log_2(2b-1) \rceil$, 
%		$h := \lceil 8b^2\log_2 k + \log_2 |V(\mathcal{C}')|\rceil$, 
%		so as to satisfy the latter inequality. Plugging this into the former inequality, we see that it is enough to have
%		$k \geq |V(\mathcal{C}')| + \lceil 8b^2\log_2 k + \log_2 |V(\mathcal{C}')|\rceil \cdot (4b-3) + 4b-1$. 
%		To this end, first recall that by \eqref{eq:size_of_C'}, there is an absolute constant $C_0 \geq 1$ such that $|V(\mathcal{C}')| \leq C_0 \cdot ah^2b^3$. 
%		Now, set $h := 20C_0 \cdot b^2 \geq 2b-1$ and $d := ab^2h = 20C_0 \cdot ab^4$. For these choices of $h$ and $d$, we have $|V(\mathcal{C}')| \leq 400C_0^3 \cdot a b^7$ and 
	This is the point where we choose the value of $k$; set 
	$k := 12b^2(4g+3)^2(a+3)(b+1)$, noting that $k = O(ab^7)$ because $g = 4b^2$.
		Both inequalities in \eqref{eq:choice_of_h} follow from \eqref{eq:size_of_C'} and our choice of $k$ and $d$. Indeed, we have:
%		\begin{align*}
%		|V(\mathcal{C}')| + (h+1) \cdot d \cdot (2b-2) + h + 1 &\leq 
%		400C_0^3 \cdot a b^7 + 
%		(20C_0 \cdot b^2 + 1) \cdot (20C_0 \cdot ab^4 \cdot (2b-2) + 1) 
%		\\ &=
%		1200 C_0 \cdot  
%		\end{align*} 
		\begin{align*}
		|V(\mathcal{C}')| + (4g+3) \cdot (d(2b-2) + 1) + d 
		&\leq 
		|V(\mathcal{C}')| + (4g+3) \cdot 2db
		\\ &\leq 
%		2(4g+3)^2(2b-1)^2(a+3)(b+1) + (4g+3) \cdot 2db
		8b^2(4g+3)^2(a+3)(b+1) + (4g+3) \cdot 2db
		\\ &\leq 
		12b^2(4g+3)^2(a+3)(b+1) = k,  
		\end{align*}
		and 
		\begin{align*}
		|V(\mathcal{C}')| \cdot k^{2g} &\leq 
		8b^2(4g+3)^2(a+3)(b+1) \cdot k^{2g} \\ &= 
		8 \cdot 12^{2g} \cdot b^{4g+2}(4g+3)^{4g+2}(a+3)^{2g+1}(b+1)^{2g+1} 
		\\ &= 
		8 \cdot 12^{2g} \cdot 2^{-4g-2} \cdot (a+3)^{-2g-1}(b+1)^{-2g-1} \cdot d^{4g+2} 
		\\ &=
%		8 \cdot 12^{2g} \cdot 2^{-4g-2} \cdot 4^{-2g-1} \cdot 2^{-2g-1} \cdot d^{4g+2} 
		2 \cdot 3^{2g} \cdot (a+3)^{-2g-1}(b+1)^{-2g-1} \cdot d^{4g+2} 
		< d^{4g+2}.
		\end{align*}
		
%		To this end, first recall that by \eqref{eq:size_of_C'}, there is an absolute constant $C_2$ such that $|V(\mathcal{C}')| \leq C_2 \cdot ah^2b^3$. 
%		Now, recall that by \eqref{eq:size_of_C'}, there is an absolute constant $C_3$ such that $|V(\mathcal{C}')| \leq C_3ah^2b^3$. 
%		So we see that if $k := C \cdot \dots$ for large enough $C$, then $k \geq |V(\mathcal{C}')| + (h+1) \cdot d \cdot (2b-2) + h + 1$, as required by the first inequality in \eqref{eq:choice_of_h}. 
%		
%		Thus, if $C_2$ is large enough then $(h+1) \cdot d \cdot (2b-2) \geq |V(\mathcal{C}')|$.  
%		Thus, and since $g=4b^2$, it is sufficient for $h$ to satisfy
%		$k \geq 4C'ah^2b^3 + h(4b-3) + 4b - 1$ and 
%		$4C'ah^2b^3 \cdot k^{8b^2} \leq 2^h$ (as clearly $2^h \leq (2^{h+1}-1)(4b-3)$). For the former inequality, it is enough to have $k \geq C''ah^2b^3$ (for some suitable absolute constant $C''$). Choosing $h := \lfloor\sqrt{k/(C''ab^3)}\rfloor$, it remains to verify the inequalities
%		$$
%		4C'ah^2b^3 \cdot k^{8b^2} \le \frac{4C'}{C''} \cdot k^{8b^2+1} \leq 2^{\lfloor\sqrt{k/(C''ab^3)}\rfloor}, \text{ and }h \ge 2b+1,
%		$$
%		which indeed hold if $k \ge Cab^7\log^2(ab)$ for some suitably large absolute constant $C$. 
		
%		Fix a choice of $h$ for which \eqref{eq:choice_of_h} holds. 
		Claims 3 and 4 imply that $D'$ satisfies Conditions 1 and 2 in Lemma \ref{lem:III_gadget_embedding}, respectively, 
%		with respect to $v := v_m$ 
		with the role of $v$ played by $v_m$, 
		and with the parameters $h$ and $d$ chosen above. 
		By Lemma \ref{lem:III_gadget_embedding}, $D'$ contains a type-III gadget $G$ and a dipath $P_0$ from $v_m$ to $p(G)$ such that $V(P_0) \cap V(G) = \{p(G)\}$, $|V(G)| \leq (2h+2)(2b-1) = (8g+6)(2b-1)$ and $|V(P_0)| \leq h(2b-1) = (4g+2)(2b-1) \leq (4g+3)(2b-1) - 1$. 
%		$(2h+2)(2b-1)$ and $|V(P_0)| \leq h(2b-1)$. 
		Now, let $\mathcal{C}_1$ be the chain formed by appending to $\mathcal{C}$ the dipath $P_0$ and the gadget $G$; so the spine of $\mathcal{C}_1$ is $P \circ P_0 \circ (p(G),q(G))$, $A_1(\mathcal{C}_1) = A_1(\mathcal{C}) \cup A(P_0)$ and $A_2(\mathcal{C}_1) = A_2(\mathcal{C}) \cup \{(p(G),q(G))\}$. It is easy to see that $\mathcal{C}_1$ is indeed a chain and that it satisfies Conditions (a)-(c) above. But this contradicts the maximality of $\mathcal{C}$. This final contradiction means that our initial assumption, that $D$ is a counterexample to Theorem \ref{thm:cycle_orientation_main}, was false. This completes the proof of the theorem.
	\end{proof}

\section{Oriented cycles with two blocks}\label{sec:twoblocks}
In this section, we prove Theorem~\ref{thm:betterbound}. We will repeatedly use the following observation:
\begin{lemma}\label{lem:fork}
	Let $\ell_1, \ell_2 \in \mathbb{N}$, and $D$ a digraph with $\delta^+(D) \ge \ell_1+\ell_2$. Then for every $v \in V(D)$, there are dipaths $P_1$ and $P_2$ in $D$ of length $\ell_1$ and $\ell_2$, respectively, which start in $v$ and satisfy $V(P_1) \cap V(P_2)=\{v\}$.
\end{lemma}
\begin{proof}
	Greedily build two disjoint dipaths starting at $v$ by attaching out-neighbors at their ends until they have lengths $\ell_1$ and $\ell_2$, respectively. 
\end{proof}
\begin{proof}[Proof of Theorem~\ref{thm:betterbound}]
	Let $D$ be an arbitrary digraph such that $\delta^+(D) \ge k_1+3k_2-5$. We have to show that there exist two internally vertex-disjoint dipaths in $D$ which start and end in the same vertices, one of length at least $k_1$, the other of length at least $k_2$. Throughout the proof, we will say that a dipath $P$ in $D$ with terminal vertex $x$ is \emph{$k_2$-good} if there exist dipaths $P_1$ and $P_2$ of length $k_2-1$ starting at $x$ such that $V(P_1) \cap V(P_2)=\{x\}$ and $V(P_i) \cap V(P)=\{x\}$ for $i \in \{1,2\}$. Note that $D$ contains a $k_2$-good dipath of positive length. Indeed, choose some arbitrary vertex $u \in V(D)$ and some out-neighbor $v$ of $u$. Since $\delta^+(D-u)\ge\delta^+(D)-1 \ge k_1+3k_2-6 \ge (k_2-1)+(k_2-1)$, we can apply Lemma~\ref{lem:fork} with $\ell_1:=\ell_2:=k_2-1$ to the vertex $v$ in the digraph $D-u$ to infer that $P:=(u,v)$ is a $k_2$-good dipath.
%	Since at least one $k_2$-good dipath of positive length exists in $D$, there must be a longest $k_2$-good dipath $P_0$ of positive length in $D$. 
	Let $P_0$ be a longest $k_2$-good dipath in $D$. We have just shown that $|P_0| > 0$. 
	Denote by $x$ the end-vertex of $P_0$ and by $P_1, P_2$ two dipaths of length $k_2-1$ starting in $x$ such that $V(P_i) \cap V(P_j)=\{x\}$ for $i \neq j \in \{0,1,2\}$.
	Let $a$ be the terminus of $P_1$ and $b$ the terminus of $P_2$. 
	\paragraph{Claim 1.} There exist dipaths $P_a, P_b$ starting in $a, b$ and ending in vertices $a', b' \in V(P_0) \setminus \{x\}$, respectively, such that $P_a$ and $P_b$ are internally vertex-disjoint from $V(P_0) \cup V(P_1) \cup V(P_2)$.
	\begin{proof}
		We prove the existence of $P_a$ and $a'$; the proof for the existence of $P_b$ and $b'$ is completely analogous. Let $D':=D-((V(P_1) \cup V(P_2))\setminus \{a\})$. Note that since $|(V(P_1) \cup V(P_2)) \setminus \{a\}|=2k_2-2$, we have $\delta^+(D') \ge \delta^+(D)-(2k_2-2) \ge k_1+k_2-3 \ge 2k_2-3$. Let $R \subseteq V(D')$ be the set of vertices reachable from $a$ by a dipath in $D'$. We claim that $R \cap (V(P_0)\setminus \{x\}) \neq \emptyset$. Suppose towards a contradiction that $R \cap (V(P_0)\setminus \{x\})=\emptyset$. Since for every vertex $r \in R$ we have $N_{D'}^+(r) \subseteq R$, we see that 
		$\delta^+(D'[R]) \geq \delta^+(D') \ge 2k_2-3$. We can now apply Lemma~\ref{lem:fork} to the vertex $a$ of $D'[R]$ with $\ell_1:=k_2-1, \ell_2:=k_2-2$ and find that $D'[R]$ contains dipaths $P_1'$ and $P_2'$ of lengths $\ell_1$ and $\ell_2$, respectively, which start at $a$ and satisfy $V(P_1') \cap V(P_2')=\{a\}$. Let $w$ be the end-vertex of the path $P_2'$. We have $N_D^+(w) \cap (V(P_0) \setminus \{x\}) \subseteq R \cap (V(P_0) \setminus \{x\})=\emptyset$. Since $d_D^+(w) \ge k_1+3k_2-5>3k_2-4=|V(P_1) \cup V(P_1') \cup (V(P_2')\setminus \{w\})|$, there must exist $w' \in N_D^+(w)\setminus(V(P_1) \cup V(P_1') \cup V(P_2'))$. Now the dipaths $P_1'$ and $P_2'\circ(w,w')$ are of length $k_2-1$, have only the starting vertex $a$ in common and are disjoint from the set $(V(P_0) \cup V(P_1))\setminus\{a\}$. Hence, $P_0 \circ P_1$ is a $k_2$-good dipath in $D$ which is longer than $P_0$, a contradiction. This shows that indeed $R \cap (V(P_0) \setminus \{x\})\neq \emptyset$. Hence, by the definition of $R$, there is a shortest dipath $P_a$ from $a$ to $R \cap (V(P_0) \setminus \{x\})$ in $D'[R]$. Write $V(P_a) \cap (V(P_0) \setminus \{x\}) =:\{a'\}$. Now $P_a$ and $a'$ satisfy the claimed properties.
	\end{proof}
	Let $A, B \subseteq V(P_0) \setminus \{x\}$ be the sets of vertices on $P_0-x$ reachable from $a,b$, respectively, by a dipath which is internally vertex-disjoint from $V(P_0) \cup V(P_1) \cup V(P_2)$. By the previous claim we have $A, B \neq \emptyset$. Let $a^\ast$ respectively $b^\ast$ denote the vertex in $A$ respectively $B$ whose distance from $x$ on $P_0$ is maximum. By symmetry, we may assume without loss of generality that $\text{dist}_{P_0}(a^\ast,x) \ge \text{dist}_{P_0}(b^\ast,x)$. Hence, $B \subseteq V(P_0[a^\ast,x])$. Fix some dipath $P_{a^\ast}$ from $a$ to $a^\ast$ in $D$ which is internally disjoint from $V(P_0) \cup V(P_1) \cup V(P_2)$. Set $Q:=P_1 \circ P_{a^\ast}$, and note that $|Q| = |P_1| + |P_{a^\ast}| = k_2 - 1 + |P_{a^\ast}| \geq k_2$. Let $Q' \subseteq Q$ be defined as follows: if the length of $Q$ is at most $k_1$ then $Q':=Q$, and otherwise $Q'$ is the unique subpath of $Q$ which starts at $x$ and has length exactly $k_1$. In the following, let $r$ denote the length of $Q'$. 
%	In each case, we have $k_2 \le r \le k_1$ and $P_1 \subseteq Q$. 
	Observe that $r = |Q'| = \min\{|Q|,k_1\}$, and hence $k_2 \leq r \leq k_1$. Moreover, $P_1 \subseteq Q'$ because $P_1$ consists of the first $k_2$ vertices of $Q$. 
	Let $y \in V(Q)$ be the terminus of $Q'$, and let us define $B^\ast$ as the subset of $B$ consisting of those vertices in $B \subseteq V(P_0-x)$ which are reachable from $b$ by a dipath which is internally vertex-disjoint from $V(P_0) \cup V(Q) \cup V(P_2)$.
	\paragraph{Claim 2.} We either have $|B^\ast| \ge k_1-r+1$, or there exists a dipath starting at $b$ and ending in $V(Q)$ which is vertex-disjoint from 
%	$(V(P_0) \setminus \{x\}) \cup (V(Q')\setminus \{y\}) \cup (V(P_2) \setminus \{b\})$.
	$(V(P_0) \cup V(Q') \cup V(P_2)) \setminus \{b,y\}$.
	\begin{proof}
		Suppose towards a contradiction that $|B^\ast| \le k_1-r$ but there exists no dipath starting at $b$ and ending in $V(Q)$ which is vertex-disjoint from 
%		$(V(P_0) \setminus \{x\}) \cup (V(Q')\setminus \{y\}) \cup (V(P_2) \setminus \{b\})$. 
		$(V(P_0) \cup V(Q') \cup V(P_2)) \setminus \{b,y\}$. 
		Let us consider the digraph 
%		$D'':=D-((V(P_0) \setminus \{x\}) \cup (V(Q') \setminus \{y\}) \cup (V(P_2)\setminus\{b\}))$. 
		$D'':=D-((V(P_0) \cup V(Q') \cup V(P_2)) \setminus \{b,y\})$. 
		Let $R \subseteq V(D'')$ denote the set of vertices reachable from $b$ in $D''$. By our assumption, we have $R \cap V(Q)=\emptyset$, and hence $R \cap (V(P_0) \cup V(Q) \cup V(P_2))=\{b\}$ (since $R \subseteq V(D'')$ and by the definition of $D''$). We claim that $N_D^+(u) \cap (V(P_0) \setminus \{x\}) \subseteq B^\ast$ for all $u \in R$. Indeed, let $u \in R$ and $v \in N_D^+(u) \cap (V(P_0) \setminus \{x\})$. By definition, there exists a $b$-$u$-dipath $P_u$ in $D''$, and $V(P_u) \subseteq R$. Then the dipath $P_u \circ (u,v)$ starts at $b$, ends in $V(P_0)\setminus\{x\}$ and is internally vertex-disjoint from $V(P_0) \cup V(Q) \cup V(P_2)$, certifying that $v \in B^\ast$. 
		
		Since $|(V(Q') \cup V(P_2))\setminus\{y,b\}|=r+k_2-2$, for every $u \in R$ we have: $$d_{D''}^+(u)\ge d_D^+(u)-|N_D^+(u) \cap (V(P_0) \setminus \{x\})|-|(V(Q') \cup V(P_2))\setminus\{y,b\}|$$ $$\ge k_1+3k_2-5-|B^\ast|-(r+k_2-2) \ge 2k_2-3,$$
		where in the last inequality we used our assumption that $|B^\ast| \le k_1-r$.
		As $N_{D''}^+(u) \subseteq R$ for every $u \in R$ (by the definition of $R$), we get that $\delta^+(D''[R]) \ge 2k_2-3$.
		Applying Lemma~\ref{lem:fork} to the vertex $b$ in $D''[R]$ with $\ell_1:=k_2-1, \ell_2:=k_2-2$, we find dipaths $P_1''$ and $P_2''$ in $D''[R]$ starting at $b$ of lengths $\ell_1$ and $\ell_2$, respectively, such that $V(P_1'') \cap V(P_2'')=\{b\}$. By the definition of $D''$ we have  
		$V(P''_i) \cap (V(P_0) \cup V(P_2)) = \{b\}$ for every $i=1,2$. Let $z$ denote the end-vertex of $P_2''$. We have $|V(P_2) \cup V(P_1'') \cup (V(P_2'')\setminus \{z\})|=3k_2-4$ and 
		$|N_D^+(z) \cap (V(P_0) \setminus \{x\})| \le |B^\ast| \le k_1-r \le k_1-k_2$. Here we used the fact that $z \in R$ and hence 
		$N_D^+(z) \cap (V(P_0) \setminus \{x\}) \subseteq B^\ast$. So we see that
		$$|N_D^+(z) \setminus (V(P_0) \cup V(P_2) \cup V(P_1'') \cup V(P_2''))| \ge k_1+3k_2-5-(k_1-k_2)-(3k_2-4)=k_2-1>0.$$
		Let $z' \notin V(P_0) \cup V(P_2) \cup V(P_1'') \cup V(P_2'')$ be an out-neighbor of $z$. The two dipaths $P_1''$ and $P_2'' \circ (z,z')$ start at $b$ and have length $k_2-1$ each. Moreover, the three dipaths $P_1''$, $P_2'' \circ (z,z')$ and $P_0 \circ P_2$ intersect each other only in the vertex $b$. Hence, $P_0 \circ P_2$ is a $k_2$-good dipath in $D$ which is strictly longer than $P_0$, a contradiction. This contradiction shows that our initial assumption was wrong, concluding the proof of Claim 2.
	\end{proof}
%	\noindent
	We will now show how to find a subdivision of $C(k_1,k_2)$ in $D$ using Claim 2. Consider the two alternatives in the conclusion of this claim. 
	The first case is that $|B^\ast| \ge k_1-r+1$. Since $B^\ast \subseteq B \subseteq V(P_0[a^\ast,x])$, this clearly implies that there exists a vertex $b^\ast \in B$ whose distance from $a^\ast$ on the dipath $P_0$ is at least $k_1-r$. By definition of $B^\ast$, there exists a dipath $P_{b^\ast}$ in $D$ starting in $b$ and ending at $b^\ast$ which is internally disjoint from $V(P_0) \cup V(Q) \cup V(P_2)$. Now the two dipaths $Q \circ P_0[a^\ast,b^\ast]$ and $P_2 \circ P_{b^\ast}$ in $D$ both start at $x$ and end at $b^\ast$, are internally vertex-disjoint, and have lengths $|Q|+|P_0[a^\ast,b^\ast]| \ge r+k_1-r=k_1$ and $|P_2|+|P_{b^\ast}| \ge k_2-1+1=k_2$, respectively. Hence, they form a subdivision of $C(k_1,k_2)$.
	
	The second case is that there exists a dipath in $D$ starting at $b$ and ending in $V(Q)$, which is vertex-disjoint from 
%	$(V(P_0) \setminus \{x\}) \cup (V(Q')\setminus \{y\}) \cup (V(P_2) \setminus \{b\})$. 
	$(V(P_0) \cup V(Q') \cup V(P_2)) \setminus \{b,y\}$.
	Let $P^\ast$ be a shortest such dipath, and let $q \in V(Q)$ denote its end-vertex. Then clearly $V(P^\ast) \cap V(Q)=\{q\}$, as well as $q \notin V(Q')\setminus \{y\}$ and $q \neq a^\ast$ (as $a^* \in V(P_0)$). This readily implies that $Q' \neq Q$, and hence by definition of $Q'$ we conclude that $Q'$ has length exactly $k_1$. Let us consider the two dipaths $Q[x,q]$ and $P_2 \circ P^\ast$ in $D$, which both start in $x$ and end in $q$. These two dipaths are internally vertex-disjoint, and have lengths $|Q[x,q]| \ge |Q'|=k_1$ and $|P_2| + |P^\ast| \ge k_2-1+1=k_2$, respectively. Hence, they form a subdivision of $C(k_1, k_2)$ in $D$.
	
	Summarizing, we have shown that $D$ contains a subdivision of $C(k_1,k_2)$ in all the cases, which concludes the proof of the theorem.
\end{proof}

\section{Subdivisions of $\bivectwo{K_3}-e$}\label{sec:k3-e}
In this section we give a proof of Theorem~\ref{thm:K3-e}. 
As it turns out, it is convenient to prove the following slightly stronger result, which clearly implies that $\text{mader}_{\delta^+}(\bivec{K}_3-e)=2$.
\begin{proposition}\label{prop:K_3-e}
Let $D$ be a digraph and $v_0 \in V(D)$ such that $d^+(v_0) \ge 1$ and $d^+(v) \ge 2$ for every $v \in V(D)\setminus \{v_0\}$. Then $D$ contains a subdivision of $\bivec{K}_3-e$.
\end{proposition}
\begin{proof}
Suppose towards a contradiction that the claim is false, and let $D$ be a counterexample which minimizes $|V(D)|$ with first priority and $|A(D)|$ with second priority. Let $v_0 \in V(D)$ be a vertex such that $d^+(v_0) \ge 1$ and $d^+(v) \ge 2$ for all $v \in V(D)\setminus \{v_0\}$.
\paragraph{Claim 1.} We have $d^+(v_0)=1$ and $d^+(v)=2$ for all $v \in V(D)\setminus \{v_0\}$.
\begin{proof}
If $d^+(v_0)>1$ or $d^+(v)>2$ for some $v \in V(D) \setminus \{v_0\}$, 
%then let $a \in A(D)$ be an out-arc of $v_0$ or $v$ respectively. Considering the digraph $D-a$, we find that this is a digraph meeting the conditions of the Lemma which is smaller than $D$ but contains no subdivision of $\bivec{K}_3-e$, a contradiction to the assumed minimality of $D$.
then we may delete an arc of $D$ and be left with a digraph $D'$ which still satisfies $d^+_{D'}(v_0) \ge 1$ and $d^+_{D'}(v) \ge 2$ for every $v \in V(D)\setminus \{v_0\}$. This contradicts the assumed minimality of $D$ (as $D'$ evidently contains no subdivision of $\bivec{K}_3-e$ either). 
\end{proof}
\paragraph{Claim 2.} $D$ is strongly connected.
\begin{proof}
If not, then there is $\emptyset \neq X \subsetneq V(D)$ such that no arc of $D$ leaves $X$. Then clearly $d_{D[X]}^+(x)=d^+(x)$ for all $x \in X$, and hence $D[X]$ meets the conditions of the Lemma. But as $D[X]$ contains no subdivision of $\bivec{K}_3-e$ and is smaller than $D$, we get a contradiction to the minimality of $D$.
\end{proof}
\paragraph{Claim 3.} There exists no partition $(W,K,Z)$ of $V(D)$ such that $W, Z \neq \emptyset$, $v_0 \in  K \cup Z$, $|K| \le 1$ and there is no arc in $D$ with tail in $W$ and head in $Z$.
\begin{proof}
Suppose towards a contradiction that a partition $(W,K,Z)$ with the described properties exists. Since $D$ is strong, we must have $|K| = 1$; say $K=\{s_0\}$ for some vertex $s_0 \in V(D)$. Since $v_0 \notin W$ and since no arc of $D$ goes from $W$ to $Z$, every vertex in $W$ has out-degree $2$ in $D[W \cup \{s_0\}]$. Since $D$ is strongly connected, there must be an $s_0$-$W$-dipath $P$ in $D$. Denoting the last vertex of $P$ by $w \in W$, we note that $V(P) \setminus \{s_0,w\} \subseteq Z$. Let $D'$ be the digraph obtained from $D[W \cup \{s_0\}]$ by adding the arc $(s_0,w)$. We clearly have $d_{D'}^+(s_0) \ge 1$, as well as $d_{D'}^+(v)=2$ for every $v \in W$ by the above. Since $|V(D')|<|V(D)|$, the minimality of $D$ implies that $D'$ contains a subdivision $S'$ of $\bivec{K}_3-e$. If $S'$ does not use the arc $(s_0,w)$ then $S' \subseteq D$. And otherwise, the subdigraph $S \subseteq D$ of $D$ defined by $V(S):=V(S') \cup V(P)$, $A(S):=(A(S') \setminus \{(s_0,w)\}) \cup A(P)$ forms a subdivision of $\bivec{K}_3-e$ in $D$. In both cases we obtain a contradiction to our assumption that $D$ does not contain a subdivision of $\bivec{K}_3-e$. This concludes the proof of the claim.
\end{proof}
In the following, let $v_1 \in V(D)$ denote the unique out-neighbor of $v_0$. The rest of the proof is divided into two cases depending on whether $v_0$ and $v_1$ have common in-neighbors.

\textbf{Case 1.} $N^-(v_0) \cap N^-(v_1)=\emptyset$. Since $d^+(v_1)=2$, there exists $v_2 \in N^+(v_1)\setminus \{v_0\}$. Let $D'$ be the digraph obtained from $D-v_1$ by adding the arc $(v_0,v_2)$ and the arcs $(x,v_0)$ for all $x \in N_D^-(v_1) \setminus \{v_0\}$. We clearly have $d_{D'}^+(v_0)=1$ and $d_{D'}^+(v)=2$ for all $v \in V(D') \setminus \{v_0\}$, since no vertex in $D$ has arcs to both $v_0$ and $v_1$. Since $|V(D')|<|V(D)|$, there must be a subdivision $S'$ of $\bivec{K}_3-e$ contained in $D'$. If $v_0 \notin V(S')$, then $S'$ is a subdigraph of $D$, which contradicts our assumption that $D$ contains no ($\bivec{K}_3-e$)-subdivision. Hence we must have $v_0 \in S'$. Since $v_2$ is the only out-neighbor of $v_0$ in $D'$, we must have $d_{S'}^+(v_0)=1$ and $(v_0,v_2) \in A(S')$. We now distinguish between two subcases depending on the in-degree of $v_0$ in $S'$. Note that every vertex of $\bivec{K}_3-e$ has in-degree either $1$ or $2$. Hence, $d_{S'}^-(v_0) \in \{1,2\}$.  

\textbf{Case 1(a).} $d_{S'}^-(v_0)=1$. Let $x_0 \in N_{D'}^-(v_0)$ be the unique in-neighbor of $v_0$ in $S'$. By definition of $D'$, we must have either $x_0 \in N_D^-(v_0) \setminus \{v_1\}$ or $x_0 \in N_D^-(v_1) \setminus \{v_0\}$. Define a subdigraph $S \subseteq D$ of $D$ as follows: If $x_0 \in N_D^-(v_0) \setminus \{v_1\}$, then we put
$V(S):=V(S') \cup \{v_1\}$ and $A(S):=(A(S') \setminus \{(v_0,v_2)\})\cup \{(v_0,v_1),(v_1,v_2)\}$, and if $x_0 \in N_D^-(v_1) \setminus \{v_0\}$, then we put $V(S):=(V(S') \setminus \{v_0\}) \cup \{v_1\}$, $A(S):=(A(S') \setminus \{(x_0,v_0),(v_0,v_2)\}) \cup \{(x_0,v_1),(v_1,v_2)\}$. It is easy to see that in each case $S$ is isomorphic to a subdivision of $S'$, and hence forms a subdivision of $\bivec{K}_3-e$ contained in $D$, a contradiction to our assumption on $D$.

\textbf{Case 1(b).} $d_{S'}^-(v_0)=2$. Let $x_1,x_2 \in N_{D'}^-(v_0)$ be the two in-neighbors of $v_0$ in $S'$. By definition of $D'$, we have 
%$x_i \in (N^-(v_0) \setminus \{v_1\}) \cup (N^-(v_1) \setminus \{v_0\})$
$x_i \in N_D^-(v_0) \setminus \{v_1\}$ or $x_i \in N_D^-(v_1) \setminus \{v_0\}$ 
for each $i=1,2$. Let us define a subdigraph $S \subseteq D$ of $D$ as follows.
Firstly, if $x_1, x_2 \in N_D^-(v_0) \setminus \{v_1\}$, then we set $V(S):=V(S') \cup \{v_1\}$ and $A(S):=(A(S')\setminus \{(v_0,v_2)\}) \cup \{(v_0,v_1),(v_1,v_2)\}$. Secondly, if $x_i \in N_D^-(v_0) \setminus \{v_1\}$ and $x_{3-i} \in N_D^-(v_1)\setminus \{v_0\}$ for some $i \in \{1,2\}$, then we set $V(S):=V(S') \cup \{v_1\}$ and $A(S):=(A(S')\setminus \{(v_0,v_2),(x_{3-i},v_0)\}) \cup \{(v_0,v_1),(v_1,v_2),(x_{3-i},v_1)\}$. Lastly, if $x_1,x_2 \in N_D^-(v_1)\setminus \{v_0\}$ then we set $V(S):=(V(S') \setminus \{v_0\}) \cup \{v_1\}$ and $A(S):=(A(S')\setminus \{(v_0,v_2),(x_1,v_0),(x_2,v_0)\})\cup\{((v_1,v_2),(x_1,v_1),(x_2,v_1)\}$. It is easy to check that in each of the three cases, $S$ is isomorphic to a subdivision of $S'$, and hence forms a subdivision of $\bivec{K}_3-e$ which is contained in $D$. This contradiction to our initial assumption on $D$ rules out Case 1.

\textbf{Case 2.} There exists a vertex $z_0 \in N^-(v_0) \cap N^-(v_1)$. Let now $A:=\{v_0,z_0\}$ and apply Theorem~\ref{setmenger} to the vertex $v_1$ versus the set $A$ in $D$. We conclude that either there are two $v_1$-$A$-dipaths intersecting only at $v_1$, or there is a set $K \subseteq V(D)\setminus \{v_1\}$ such that $|K| \le 1$ and there is no dipath in $D-K$ starting in $v_1$ and ending in $A$. 

In the first case, let $P_1$ and $P_2$ be dipaths such that $V(P_1) \cap V(P_2)=\{v_1\}$ and such that $P_1$ ends in $v_0$, while $P_2$ ends in $z_0$. Now the subdigraph $S \subseteq D$ with vertex set $V(S):=V(P_1) \cup V(P_2)$ and arc-set $A(S):=A(P_1) \cup A(P_2) \cup \{(v_0,v_1),(z_0,v_0),(z_0,v_1)\}$ forms a subdivision of $\bivec{K}_3-e$ with branch vertices $v_0, v_1, z_0$. This is a contradiction to our initial assumption on $D$. 

In the second case, let $W \subseteq V(D)-K$ be the subset of vertices reachable from $v_1$ by a dipath in $D-K$ and let $Z:=V(D) \setminus (W \cup K)$. Since there is no $v_1$-$A$-dipath in $D-K$, we must have $v_0 \in A \subseteq K \cup Z$. We further have $v_1 \in W$ and $A \setminus K \subseteq Z$, hence $W, Z\neq \emptyset$. Moreover, by definition of $W$, no arc in $D$ starts in $W$ and ends in $Z$. All in all, this shows that the partition $(W,K,Z)$ of $V(D)$ yields a contradiction to Claim 3. 

Since we arrived at contradictions in all possible cases, we conclude that our initial assumption about the existence of $D$ was wrong. This completes the proof of Proposition \ref{prop:K_3-e}. 
\end{proof}

\section{Subdivisions and arc-connectivity}\label{sec:arcconn}
In this section we prove Propositions~\ref{prop:arc_connectivity} and~\ref{prop:arc_connectivity2}, showing that $\bivec{K}_4$ and $\bivec{S}_4$ are not $\kappa'$-maderian.
%\begin{theorem}
%For every $k \in \mathbb{N}$ there exists a digraph $H_k$ with $\kappa'(H_k) \ge k$ such that $H_k$ contains no subdivision of $\bivec{K}_4$.
%\end{theorem}
\begin{proof}[Proof of Proposition \ref{prop:arc_connectivity}]
A construction of Thomassen~\cite{thom85} shows that for every integer $k \ge 1$, there exists a digraph $D_k$ such that $\delta^+(D_k)=k$ and $D_k$ contains no directed cycle of even length. For every $k \ge 1$ let $\revvec{D}_k$ denote the digraph obtained from $D_k$ by reversing all its arcs. Then clearly we have $\delta^-(\revvec{D}_k)=k$. Let $G_k'$ be the digraph obtained from the vertex-disjoint union of a copy of $D_k$ with vertex-set $A$ and a copy of $\revvec{D}_k$ with vertex-set $B$ by adding all the arcs in $B \times A$ (i.e., all arcs from $B$ to $A$). Note that since $|A|=|B|=|V(D_k)|>k$, we have $\delta^+(G_k')=\delta^-(G_k')=k$. Finally, we define $G_k$ as the digraph obtained from $G_k'$ by adding a vertex $v \notin V(G_k')$ as well as all arcs $(v,x), (x,v)$ for $x \in V(G_k')$. We claim that $G_k$ is strongly $k$-arc-connected. Indeed, let $E \subseteq A(G_k)$ be a set of arcs such that $|E|<k$. We claim that in $G_k-E$, every vertex $x \in V(G_k')$ can reach and is reachable from $v$ via a dipath. This will show that $G_k-E$ is strongly connected, as required. Let $x \in V(G_k')$ be given arbitrarily, and let $x_1,\ldots,x_k \in V(G_k')$ be $k$ pairwise distinct out-neighbors of $x$ in $G_k'$. Consider the $k$ arc-disjoint dipaths $P_i:=(x,x_i)\circ (x_i,v)$, $i=1,\ldots,k$. At least one of these dipaths must be disjoint from $E$ and hence constitute an $x$-$v$-dipath in $D-E$. With a symmetric argument considering $k$ distinct in-neighbors of $x$, we also obtain that there is a $v$-$x$-dipath in $G_k-E$, as required.
We further claim that $G_k$ contains no subdivision of $\bivec{K}_4$. Indeed, suppose this was the case, then clearly there would be $S \subseteq G_k-v=G_k'$ such that $S$ is a subdivision of $\bivec{K}_3$. As is easy to see, $S$ must contain an even directed cycle. Since there is no arc in $G_k'$ from $A$ to $B$, we find that this cycle must be entirely contained in either $G_k'[A] \simeq D_k$ or $G_k'[B] \simeq \revvec{D}_k$. This however means that $D_k$ contains an even directed cycle, a contradiction. This contradiction shows that $G_k$ contains no subdivision of $\bivec{K}_4$, and this concludes the proof. 
\end{proof}
\begin{proof}[Proof of Proposition~\ref{prop:arc_connectivity2}]
A construction of Thomassen~\cite{thom85} shows that for every integer $k \ge 1$, there exists a digraph $R_k$ such that $\delta^+(R_k)=k$ and $R_k$ contains no subdivision of the bioriented $3$-star $\bivec{S}_3$. For $k \ge 1$, let us denote by $\revvec{R}_k$ the digraph obtained from $R_k$ by reversing all its arcs. Let $H_k'$ be the digraph obtained from the disjoint union of a copy of $R_k$ with vertex-set $A$ and a copy of $\revvec{R}_k$ with vertex-set $B$ by adding all the arcs in $B \times A$. Since $R_k$ and $\revvec{R}_k$ have at least $k$ vertices, we obtain that $\delta^+(H_k')=\delta^-(H_k')=k$. We now define $H_k$ to be the digraph obtained from two disjoint copies of $H_k'$ with vertex-sets $X$ and $Y$ by adding two distinct new vertices $u$ and $v$ as well as the following arcs: $(u,x)$ and $(x,v)$ for every $x \in X$, and $(y,u)$ and $(v,y)$ for every $y \in Y$. We claim that $H_k$ is strongly $k$-arc-connected. Indeed, let $E \subseteq A(H_k)$ be an arbitrarily given set of arcs such that $|E|<k$. We must prove that $H_k-E$ is strongly connected. For this, it clearly suffices to show that in $H_k-E$, every vertex in $X$ can reach $v$ and is reachable from $u$, and every vertex in $Y$ can reach $u$ and is reachable from $v$. Let $x \in X$ be any given vertex, and let $x_1^-,\ldots,x_k^- \in X$ denote $k$ distinct in-neighbors of $x$ in $H_k[X] \simeq H_k'$. Among the $k$ arc-disjoint $u$-$x$-dipaths $(u,x_i^-)\circ(x_i^-,x), i=1,\ldots,k$ in $H_k$, at least one must also exist in $H_k-E$, and hence $x$ is reachable from $u$ in $H_k-E$. Similarly, considering $k$ distinct out-neighbors $x_1^+,\ldots,x_k^+ \in X$ of $x$ in $H_k[X]$, and considering the arc-disjoint $x$-$v$-dipaths $(x,x_i^+),(x_i^+,v),i=1,\ldots,k$, we find that there is an $x$-$v$-dipath in $H_k-E$. With a symmetric argument for the vertices in $Y$, we can verify the above claim, showing that $H_k-E$ is strongly connected. This shows that indeed \nolinebreak $\kappa'(H_k) \ge k$. 

Next we claim that $H_k$ does not contain a subdivision of $\bivec{S}_4$. Suppose otherwise. Then there exists a vertex $w \in V(H_k)$ and directed cycles $C_1,C_2,C_3,C_4$ in $H_k$ such that $w \in V(C_i)$ for $i=1,\ldots,4$, and such that the sets $V(C_i)\setminus\{w\}, 1 \le i \le 4$, are pairwise disjoint. Suppose first that $w \in \{u,v\}$. Without loss of generality, we may assume that $w = u$ (the case $w = v$ is symmetric). Then for each $1 \leq i \leq 4$, $C_i - w$ is a dipath which starts in $X$ and ends in $Y$ (since the vertex of $C_i$ preceding $w = u$ must be in $Y$, and the vertex of $C_i$ succeeding $w = u$ must be in $X$). It follows that $C_i-w, 1 \le i \le 4$, are pairwise vertex-disjoint dipaths from $X$ to $Y$, 
contradicting the fact that $X$ and $Y$ can be disconnected in $H_k$ by deleting only two vertices, namely $u$ and $v$. 
Suppose now that $w \in X \cup Y$. Note that every directed cycle in $H_k$ is either contained in $H_k[X]$, or contained in $H_k[Y]$, or contains both $u$ and $v$. Hence, if $w \in X$, then at least three of the cycles $C_i, 1 \le i \le 4$, are contained in $H_k[X] \simeq H_k'$, and if $w \in Y$ then at least three of the cycles $C_i, 1 \le i \le 4$, are contained in $H_k[Y] \simeq H_k'$. So we see that in each case, $H_k'$ must contain a subdivision of $\bivec{S}_3$. Since every subdivision of $\bivec{S}_3$ is a strongly connected digraph, and since there are no arcs from $A$ to $B$ in $H_k'$, we find that this subdivision must be entirely contained in either $H_k'[A] \simeq R_k$ or $H_k'[B] \simeq \revvec{R}_k$. Since $\bivec{S}_3$ is invariant under the reversal of all arcs, we obtain that in each case $R_k$ must contain a subdivision of $\bivec{S}_3$. This contradicts our initial assumptions on the sequence $(R_k)_{k \ge 1}$. This contradiction proves the claim of the proposition; namely, $H_k$ is indeed a $k$-strongly arc connected digraph not containing $\bivec{S}_4$ as a subdivision.
\end{proof}
\section{Open Problems}
In this concluding section, we would like to mention further open problems related to subdivisions in digraphs of large minimum out-degree, which we discovered during the work on this paper. 

Theorem~\ref{thm:cycles} shows that orientations of cycles are $\delta^+$-maderian, and that for an orientation $C$ of a cycle, $\text{mader}_{\delta^+}(C)$ grows polynomially in $|C|$. Aboulker~et~al.~actually conjectured the very explicit bound of $\text{mader}_{\delta^+}(C) \le 2|C|-1$ (cf.~\cite{aboulker}, Conjecture 27). However, it is even unclear to us whether $\text{mader}_{\delta^+}(C)$ should be linear in $|C|$ at all.
\begin{problem}
Does it hold that $\text{mader}_{\delta^+}(C)=O(|V(C)|)$ for every orientation $C$ of a cycle?
\end{problem}
We remark that Theorem~\ref{thm:cycle_orientation_main} gives a positive answer to this question when the size of a longest block in $C$ is bounded by a constant.

Disjoint union is a basic graph operation under which one might naturally anticipate the $\delta^+$-maderian property to be preserved. Yet, despite quite a bit of effort, this intuition is only known to hold in a few special cases. 
Thomassen's Theorem for example states that the disjoint union of $k$ digons is $\delta^+$-maderian for all $k$. 
A common generalization of this result and Theorem~\ref{thm:cycles} would be the following.

\begin{conjecture}
Any disjoint union of orientations of cycles is $\delta^+$-maderian.
\end{conjecture}

Digraph subdivision is another graph operation under which it is plausible to expect that the $\delta^+$-maderian property is preserved.
\begin{conjecture}\label{longersubdivisions}
If a digraph $F$ is $\delta^+$-maderian, all subdivisions of $F$ are $\delta^+$-maderian as well.
\end{conjecture}
Conjecture~\ref{longersubdivisions} would follow if we could show that every digraph of large enough out-degree contains a subdivision of some digraph of out-degree $k$ in which every subdivision path is long.  
\begin{conjecture}\label{conj:>=2-subdivision}
There is a function $f:\mathbb{N} \rightarrow \mathbb{N}$ such that for every $k \in \mathbb{N}$ and for every digraph $D$ with $\delta^+(D) \ge f(k)$, there exists a digraph $D'$ such that $\delta^+(D') \ge k$ and $D$ contains a subdivision of $D'$ in which every subdivision-path has length at least two.
\end{conjecture}

An important step towards Conjecture \ref{con:forests} would be to show that attaching an out-leaf to any vertex of a $\delta^+$-maderian digraph yields still a $\delta^+$-maderian digraph.
\begin{conjecture}\label{conj:out_leaf}
If $F$ is a $\delta^+$-maderian digraph, $v_0 \in V(F)$ and $F^\ast$ is the digraph obtained from $F$ by adding a new vertex $v_1$ and the arc $(v_0,v_1)$, then $F^\ast$ is $\delta^+$-maderian as well.
\end{conjecture}
\noindent
Conjecture~\ref{conj:out_leaf} would follow directly from the following natural statement.
We call a set of vertices $X$ in a digraph $D$ \emph{in-dominating set} if every $y \in V(D) \setminus X$ has an out-neigbor in $X$.
\begin{conjecture}\label{conj:in-dominating}
There exists a function $f:\mathbb{N} \rightarrow \mathbb{N}$ such that the following holds for every $k \ge 1$. If $D$ is a digraph with $\delta^+(D) \ge f(k)$, then there exists an in-dominating set $X \subsetneq V(D)$ such that $\delta^+(D-X) \ge k$.
\end{conjecture}

Another interesting direction is to characterize the undirected graphs $F$ for which 
%$\bivec{F}$ (i.e., the biorientation of $F$) is $\delta^+$-maderian. 
the biorientation $\bivec{F}$ of $F$ is $\delta^+$-maderian. 
If $\bivec{F}$ is $\delta^+$-maderian, then $F$ must be a forest, since every bioriented cycle has arc-connectivity two and hence is not $\delta^+$-maderian (see the necessary properties of $\delta^+$-maderian digraphs mentioned in the introduction).
Furthermore, it is known that $\bivec{S_3}$ is not $\delta^+$-maderian \cite{thom85}. Thus, if $\bivec{F}$ is $\delta^+$-maderian then $F$ must be a path-forest.
Thomassen's result~\cite{thom83} shows that a biorientation of any matching is $\delta^+$-maderian. By Theorem \ref{thm:K3-e}, $\bivec{S_2}= \bivec{P_3}$ is $\delta^+$-maderian (where $P_{\ell}$ denotes the path on $\ell$ vertices).
The first open case is that of $\bivec{P_4}$.
\begin{problem}
	Is $\bivec{P_4}$ $\delta^+$-maderian?
\end{problem}

Finally, several open problems arise from the questions considered in Section~\ref{sec:arcconn}. Given that $\bivec{K}_4$ and $\bivec{S_4}$ are not $\kappa'$-maderian (see Propositions \ref{prop:arc_connectivity}-\ref{prop:arc_connectivity2}), it is natural to ask
%the following question.
whether $\bivec{K}_3$ and \nolinebreak $\bivec{S_3}$ \nolinebreak are. 
\begin{problem}
Is $\bivec{K}_3$ $\kappa'$-maderian? Is $\bivec{S_3}$ $\kappa'$-maderian?
\end{problem}
As mentioned in the introduction, every subdivision of $\bivec{K_3}$ contains an even dicycle, and one cannot force an even dicycle by means of minimum out-degree \cite{thom85}. Thus, even dicycles can be thought of as an obstacle to forcing subdivisions of $\bivec{K_3}$. Interestingly, this obstacle disappears when considering arc-connectivity (rather than out-degree), as a theorem of Thomassen~\cite{thom92} shows that every digraph $D$ with $\kappa'(D) \ge 3$ contains an even dicycle. This can be thought of as a hint that $\bivec{K}_3$ could in fact be $\kappa'$-maderian.

A critical first step towards the resolution of Problem~\ref{prob:connectivity} for vertex-connectivity is the following.
\begin{problem}
Is there a constant $K \in \mathbb{N}$ such that every $K$-strongly-vertex connected digraph contains two vertices $x \neq y$ and four pairwise internally vertex-disjoint dipaths, two from $x$ to $y$ and two from $y$ to $x$? 
\end{problem}

{\bf Acknowledgement} The research on this project was initiated during a joint research workshop of Tel Aviv University and the Freie Universit\"at Berlin on Ramsey Theory, held in Tel Aviv in March 2020, and partially supported by GIF grant G-1347-304.6/2016. We would like to thank the German-Israeli Foundation (GIF) and both institutions for their support.

\end{document}